\documentclass[11pt, reqno]{amsart}
\usepackage{amsfonts,latexsym,enumerate}
\usepackage{amsmath}
\usepackage{amscd}
\usepackage{float,amsmath,amssymb,mathrsfs,bm,multirow,graphics}
\usepackage[dvips]{graphicx}
\usepackage[percent]{overpic}
\usepackage{amsaddr}
\usepackage[numbers,sort&compress]{natbib}
\usepackage{xcolor}

\newcommand{\R}{{\Bbb R}}

\newcommand{\C}{{\Bbb C}}

\newcommand{\re}{\text{\upshape Re\,}}
\newcommand{\im}{\text{\upshape Im\,}}

\newcommand{\Ai}{\text{\upshape Ai}}

\newcommand{\IV}{\text{\upshape IV}}

\def\XXint#1#2#3{{\setbox0=\hbox{$#1{#2#3}{\int}$}
\vcenter{\hbox{$#2#3$}}\kern-.5\wd0}}

\newtheorem{theorem}{Theorem}
\newtheorem{proposition}{Proposition}[section]

\newtheorem{lemma}[proposition]{Lemma}

\newtheorem{remark}[proposition]{Remark}

\newtheorem{figuretext}{Figure}

\numberwithin{equation}{section}

\input epsf
\date{\today}
\title[Airy and Painlev\'e asymptotics  for the mKdV equation]
{Airy and Painlev\'e asymptotics \\ for the mKdV equation}

\author{C. Charlier and J. Lenells}
\address{Department of Mathematics, KTH Royal Institute of Technology, \\ 100 44 Stockholm, Sweden.}
\email{cchar@kth.se, jlenells@kth.se}

\begin{document}

\begin{abstract} 
\noindent
We consider the higher order asymptotics for the mKdV equation in the Painlev\'e sector. We first show that the solution admits a uniform expansion to all orders in powers of $t^{-1/3}$ with coefficients that are smooth functions of $x(3t)^{-1/3}$. We then consider the special case when the reflection coefficient vanishes at the origin. In this case, the leading coefficient which satisfies the Painlev\'e II equation vanishes. We show that the leading asymptotics is instead described by the derivative of the Airy function. We are also able to express the subleading term explicitly in terms of the Airy function. 
\end{abstract}

\maketitle



\setcounter{tocdepth}{1}
\tableofcontents

\section{Introduction}
The initial value problem for a nonlinear integrable evolution equation can be analyzed via the inverse scattering transform, which expresses the solution of the equation in terms of the solution of a matrix Riemann--Hilbert (RH) problem. One of the greatest advantages of this approach is that it can be used to derive detailed asymptotic formulas for the long-time behavior of the solution. In the pioneering work \cite{DZ1993}, the asymptotic behavior of the solution $u(x,t)$ of the modified Korteweg-de Vries (mKdV) equation 
\begin{align}\label{mkdv}
  u_t - 6u^2u_x + u_{xxx} = 0
\end{align}
was rigorously established via the application of a nonlinear version of the steepest descent method. Six asymptotic sectors, denoted by I-VI, were identified (see Figure \ref{mkdvsectors.pdf}), and in each sector the leading asymptotic term was determined explicitly in terms of the reflection coefficient $r(k)$. Since $r(k)$ is defined in terms of the initial data $u_0(x) = u(x,0)$ alone, this provides an effective solution of the problem. 

In Sector IV, defined by $|x| \leq M t^{1/3}$ for some constant $M > 0$, it was shown in \cite{DZ1993} that the solution $u(x,t)$ obeys the asymptotics
\begin{align}\label{uleadingsectorIV}
u(x,t) = \frac{u_1(y)}{t^{1/3}} + O(t^{-2/3}), \qquad t \to \infty,
\end{align}
where the leading coefficient $u_1(y)$ is given by
\begin{align}\label{u1expression}
u_1(y) = 3^{-1/3}u_P(y;s,0,-s), \qquad y := x(3t)^{-1/3},
\end{align}
with $u_P(y; s,0,-s)$ being a solution of the Painlev\'e II equation 
\begin{align}\label{painleve2}
  u_{P}''(y) = yu_P(y) + 2u_P(y)^3
\end{align}  
determined by the value of $s := ir(0)$. In Sector II, defined by $-M t\leq x < 0$ and $|x|t^{-1/3}\to \infty$, the solution instead asymptotes, to leading order, to a slowly decaying modulated sine wave.
Later, in \cite{DZ1994}, the same authors derived an asymptotic expansion valid to all orders as $(x,t) \to \infty$ in the subsector of Sector II given by $-M_1t \leq x \leq -M_2 t$ and recursive formulas were derived for the higher order coefficients. 
Despite this progress, it appears that the higher order asymptotics in other asymptotic regions remains to be studied in detail.

In this paper, we consider the higher order asymptotics of (\ref{mkdv}) in Sector IV. If $r(0) \neq 0$, then equation (\ref{uleadingsectorIV}) provides the leading asymptotics in this sector. However, if $r(0) = 0$, then $s = 0$ and the associated solution $u_P(y; s,0,-s)$ of Painlev\'e II vanishes identically. Thus $u_1(y) \equiv 0$ and (\ref{uleadingsectorIV}) does not provide any information on the leading term in this case.

Our first result (Theorem \ref{mainth1}) states that $u(x,t)$ admits a uniform expansion to all orders in Sector IV of the form
$$u(x,t) \sim \sum_{j=1}^\infty \frac{u_j(y)}{t^{j/3}}  \quad \text{as $t \to \infty$},$$
where $\{u_j(y)\}_1^\infty$ are smooth functions of $y \in \R$ with the leading coefficient $u_1(y)$ given by (\ref{u1expression}). In general, the higher order coefficients $\{u_j(y)\}_2^\infty$ cannot be computed explicitly, but they are solutions of a hierarchy of ordinary differential equations of which Painlev\'e II is the first member. 

Our second result (Theorem \ref{mainth2}) deals with the special case when $r(0) = 0$. Interestingly, in this case it is possible not only to compute the leading coefficient $u_2(y)$ explicitly, but also the subleading term $u_3(y)$. The coefficients are expressed in terms of the Airy function $\Ai(y)$. More precisely, if $r(0) = 0$, we show that the solution $u(x,t)$ obeys the asymptotic expansion
$$u(x,t) = \frac{u_2(y)}{t^{2/3}} +  \frac{u_3(y)}{t} + \cdots \quad \text{as $t \to \infty$}$$
uniformly for $|x| \leq M t^{1/3}$, where
\begin{align}\label{u2u3explicit}
u_2(y) = \frac{r'(0)}{2\times 3^{2/3}}\Ai'(y),
\qquad u_3(y) = -\frac{ir''(0)}{24}y\Ai(y).
\end{align}

Although we only present our results for the mKdV equation, it is clear that similar considerations apply also to other integrable equations with Painlev\'e type asymptotic sectors. 

Let us finally mention some other works that also use nonlinear steepest descent techniques for Riemann--Hilbert problems to derive asymptotic formulas in related situations. Asymptotic results for the KdV equation can be found in \cite{AELT2016, DVZ1994, EGKT2013, GT2009}; asymptotic results for the mKdV equation on the half-line are presented in \cite{BFS2004, L2016}; and higher order asymptotics in sectors analogous to Sector II but for other equations have been derived in \cite{HXF2015, V2000}.

The two main theorems are stated in Section \ref{mainsec}. Following a brief review of the RH formalism relevant for (\ref{mkdv}) in Section \ref{overviewsec}, the proofs are presented in Sections \ref{sectorIVgeqsec} and \ref{sectorIVleqsec}. A substantial part of the work---namely the constructions of appropriate local models---is deferred to the two appendices.

\begin{figure}
\begin{center}
\begin{overpic}[width=.6\textwidth]{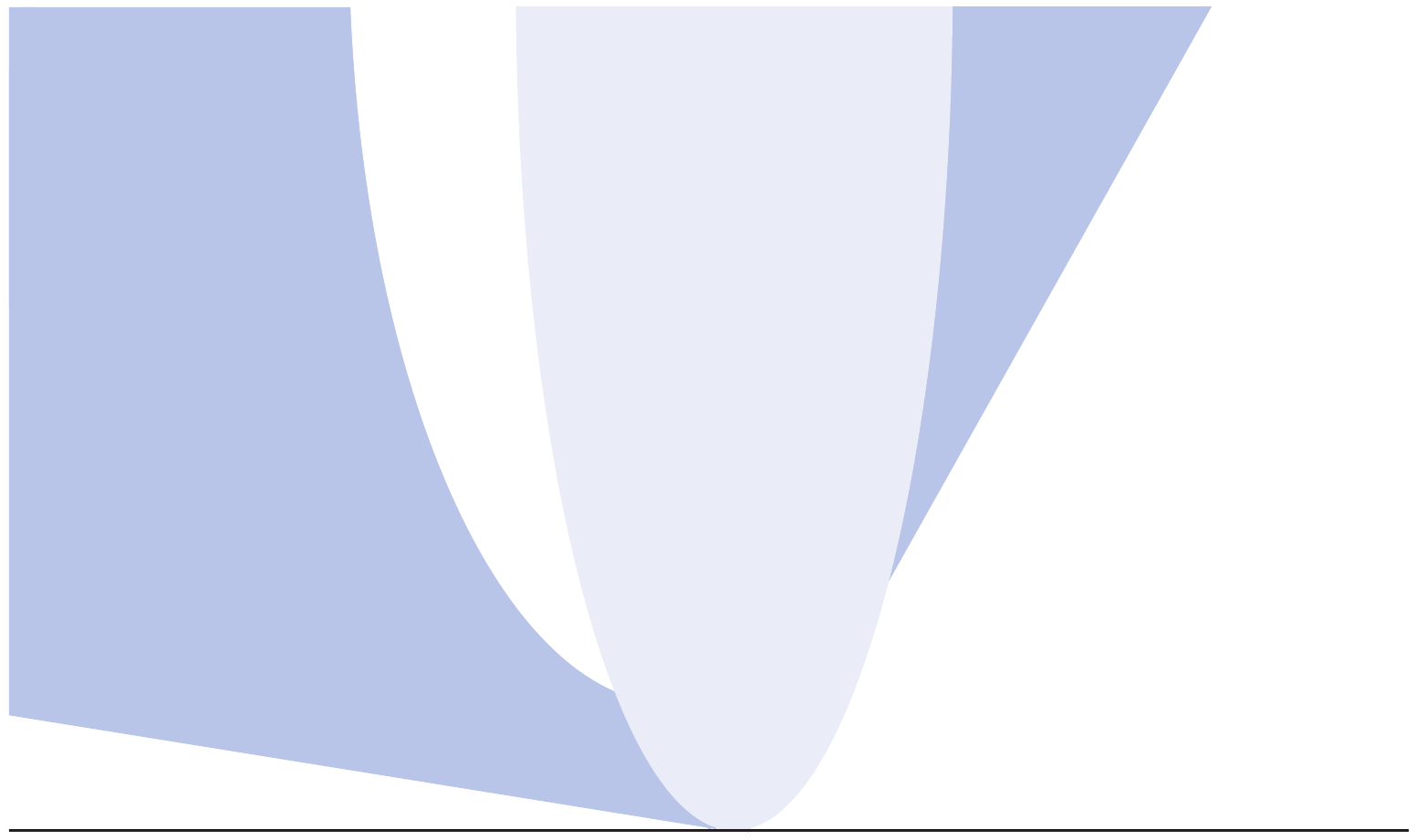}
      \put(102,0){\small $x$}
      \put(5,3){\small $I$}
      \put(10,30){\small $II$}
      \put(28,47){\small $III$}
      \put(49,47){\small $IV$}
      \put(72,47){\small $V$}
      \put(85,20){\small $VI$}
      \end{overpic}
     \begin{figuretext}\label{mkdvsectors.pdf}
        The six asymptotic sectors in the $(x,t)$-plane for the mKdV equation (\ref{mkdv}). 
     \end{figuretext}
     \end{center}
\end{figure}

\section{Main results}\label{mainsec}
Given initial data $u_0(x)$, $x \in \R$, in the Schwartz class $\mathcal{S}(\R)$, we let $r(k)$, $k \in \R$, be the reflection coefficient associated to $u_0$ according to the inverse scattering transform.\footnote{We use the same conventions for the definition of $r(k)$ as in \cite{DZ1993}.}
The following two theorems are the main results of the paper.

\begin{theorem}\label{mainth1}
Let $u_0 \in \mathcal{S}(\R)$ be a function in the Schwartz class and let $r \in \mathcal{S}(\R)$ be the associated reflection coefficient.
Then the solution $u(x,t)$ of (\ref{mkdv}) with initial data $u_0(x)$ satisfies the following asymptotic formula as $t \to \infty$:
\begin{align}\label{asymptoticsinIV}
u(x,t) = \sum_{j=1}^N \frac{u_j(y)}{t^{j/3}} + O\bigg(\frac{1}{t^{\frac{N+1}{3}}}\bigg),  \qquad |x| \leq M t^{1/3}, 
\end{align}
where 
\begin{itemize}
\item The formula holds uniformly with respect to $x$ in the given range for any fixed $M > 0$ and $N \geq 1$.
\item The variable $y$ is defined by
\begin{align*}
 y = \frac{x}{(3t)^{1/3}}.
\end{align*}

\item $\{u_j(y)\}_1^\infty$ are smooth functions of $y \in \R$.

\item The function $u_1(y)$ is given by (\ref{u1expression}) where $s =  ir(0)$ and $u_P(y; s, 0, -s)$ denotes the smooth real-valued solution of the Painlev\'e II equation (\ref{painleve2}) corresponding to $(s,0,-s)$ via (\ref{painlevebijection}).
 
\end{itemize}
\end{theorem}

\begin{theorem}\label{mainth2}
If the reflection coefficient $r(k)$ satisfies $r(0) = 0$, then the asymptotic formula (\ref{asymptoticsinIV}) holds with $u_1(y) \equiv 0$ and with $u_2(y)$ and $u_3(y)$ given explicitly by (\ref{u2u3explicit}).
\end{theorem}

\begin{remark}[Hierarchy of differential equations]\upshape
Substituting the asymptotic expansion (\ref{asymptoticsinIV}) into (\ref{mkdv}) and identifying coefficients of powers of $t^{-1/3}$, we see that $\{u_j(y)\}_1^\infty$ are smooth real-valued solutions of a hierarchy of linear ordinary differential equations. The first four members of this hierarchy are given by
\begin{align}\nonumber
& u_1'''- y u_1' - u_1 = 2 \times 3^{2/3} (u_1^3)',
	\\\nonumber
& u_2''' - y u_2' - 2u_2 = 6 \times 3^{2/3}(u_1^2 u_2)',
	\\\nonumber
& u_3''' - y u_3' - 3u_3 = 6 \times 3^{2/3}(u_1 u_2^2 + u_1^2 u_3)',
	\\\label{ujhierarchy}
& u_4''' - y u_4' -4u_4 = 2 \times 3^{2/3}(u_2^3 + 6u_1 u_2 u_3 + 3u_1^2 u_4)'.
\end{align}	
As expected, the first of these equations reduces to the Painlev\'e II equation (\ref{painleve2}) when $u_1(y)$ and $u^P(y)$ are related by (\ref{u1expression}). 
Moreover, in the special case when $r(0) = 0$, we have $s = ir(0) = 0$, so $u_1(y) \equiv 0$ and the first equations in the hierarchy are 
\begin{align}\nonumber
& u_2''' - y u_2' - 2u_2 = 0,
	\\\nonumber
& u_3''' - y u_3' - 3u_3 = 0,
	\\\nonumber
& u_4''' - y u_4' -4u_4 = 2 \times 3^{2/3}(u_2^3)'.
\end{align}	
The expressions (\ref{u2u3explicit}) for $u_2(y)$ and $u_3(y)$ are consistent with these equations, because, for each integer $j \geq 1$, 
\begin{align*}
& f(y) = \frac{d^{j-1}}{dy^{j-1}}\Ai(y), 
%\qquad
%f_2(y) = \frac{d^{j-1}}{dy^{j-1}}\Bi(y),
\end{align*}
satisfies the homogeneous linear ODE $f''' - y f' - j f = 0$ and $\Ai''(y) = y\Ai(y)$.
\end{remark}

\section{Riemann--Hilbert problem}\label{overviewsec}
According to the inverse scattering transform formalism, the solution $u(x,t)$ of the Cauchy problem for (\ref{mkdv}) on $\R$ with initial data $u_0(x)$ in the Schwartz class $\mathcal{S}(\R)$ is given by 
\begin{align}\label{recoveru}
u(x,t) %= 2\lim_{k \to \infty} (km(x,t,k))_{12}
= 2\lim_{k \to \infty} k(m(x,t,k))_{21},
\end{align}
where $m(x,t,k)$ denotes the unique solution of the RH problem
\begin{align}\label{preRHm}
\begin{cases}
\text{$m(x, t, k)$ is analytic for $k \in \C \setminus \R$,} \\
\text{$m_+(x,t,k) = m_-(x, t, k) v(x, t, k)$ for $k \in \R$,} \\
\text{$m(x,t,k) = I + O(k^{-1})$ as $k \to \infty$,} 
\end{cases}
\end{align}
with jump matrix $v(x,t,k)$ given by (cf. \cite{DZ1993})
\begin{align*}
& v(x,t,k) = \begin{pmatrix} 1 - |r(k)|^2 & -\overline{r(k)}e^{-t\Phi(\zeta,k)} \\ r(k)e^{t \Phi(\zeta,k)} & 1 \end{pmatrix}, 
	\\
& \Phi(\zeta, k) := 2i(\zeta k + 4k^3), \qquad \zeta := \frac{x}{t},
\end{align*}
and $m_+$ and $m_-$ denote the boundary values of $m$ from the upper and lower half-planes, respectively.
If $u_0 \in \mathcal{S}(\R)$, then the reflection coefficient $r(k)$ belongs to the Schwartz class and obeys the bound
$$\sup_{k \in \R} |r(k)| < 1$$
and the symmetry
\begin{align}\label{rsymm}
r(k) = -\overline{r(-k)}, \qquad k \in \R.
\end{align}
It follows that $v$ obeys the symmetries
$$v(x,t,k) = \sigma_1\overline{v(x,t,\bar{k})}^{-1}\sigma_1 = \sigma_1 \sigma_3 v(x,t,-k)^{-1}\sigma_3\sigma_1,\qquad k \in \R,$$
where 
$$\sigma_1 = \begin{pmatrix} 0 & 1 \\1 & 0 \end{pmatrix}, \quad \sigma_2 = \begin{pmatrix} 0 & -i \\ i & 0 \end{pmatrix}, \quad \sigma_3 = \begin{pmatrix} 1 & 0 \\0 & -1 \end{pmatrix}.$$
Hence, by uniqueness of the solution of the RH problem (\ref{preRHm}), $m$ satisfies
\begin{align}\label{msymm}
m(x,t,k) = \sigma_1\overline{m(x,t,\bar{k})}\sigma_1 = \sigma_1 \sigma_3 m(x,t,-k)\sigma_3\sigma_1, \qquad k \in \C \setminus \R.
\end{align}

\begin{figure}
\begin{center}
\begin{overpic}[width=.4\textwidth]{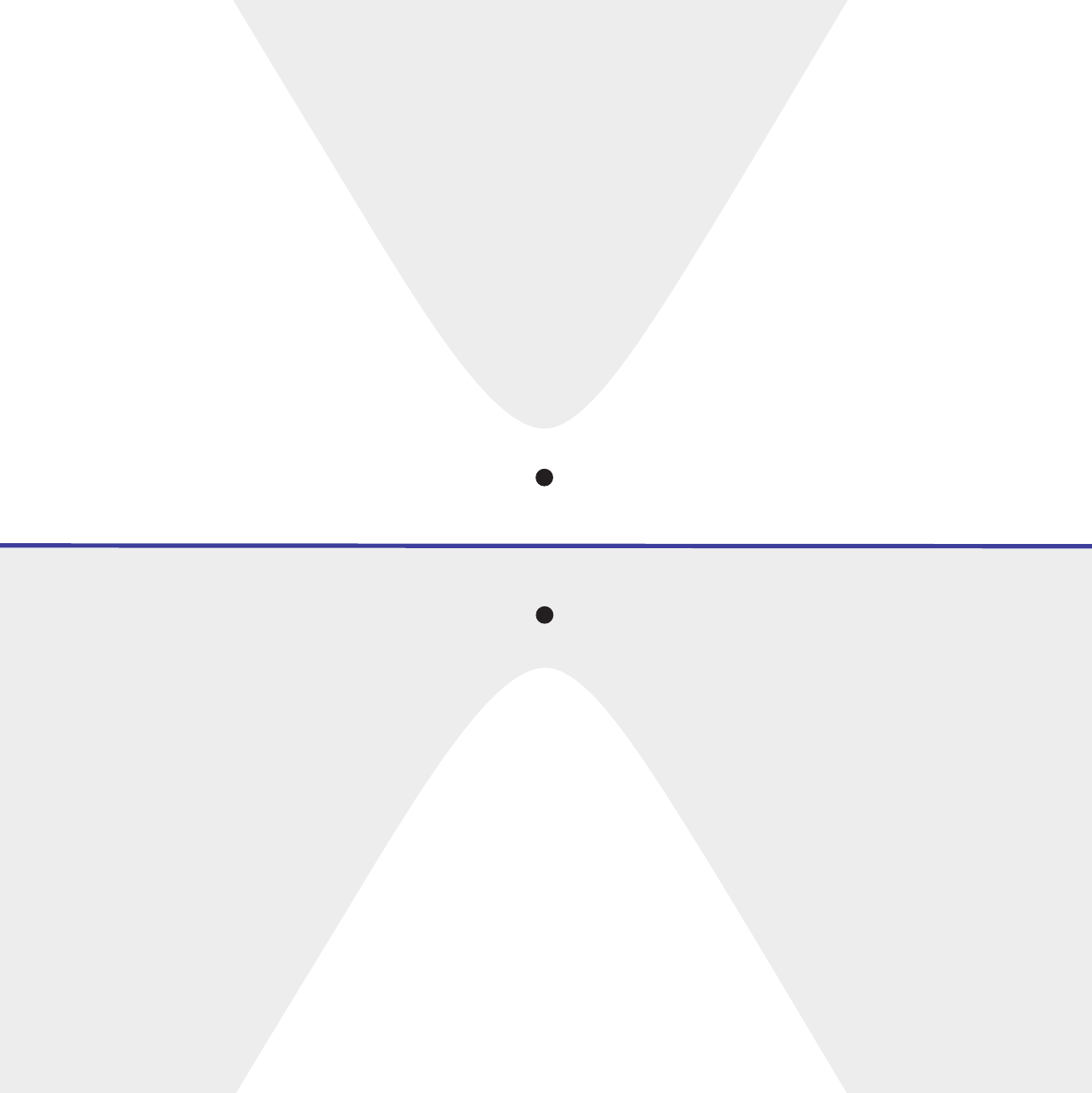}
      \put(102,48){\small $\R$}
      \put(42.2,55){\small $k_0$}
      \put(37.5,42){\small $-k_0$}
      \put(38,85){\small $\re \Phi > 0$}
      \put(66,61){\small $\re \Phi < 0$}
      \put(66,34){\small $\re \Phi > 0$}
      \put(38,10){\small $\re \Phi < 0$}
    \end{overpic}\qquad\qquad
    \begin{overpic}[width=.4\textwidth]{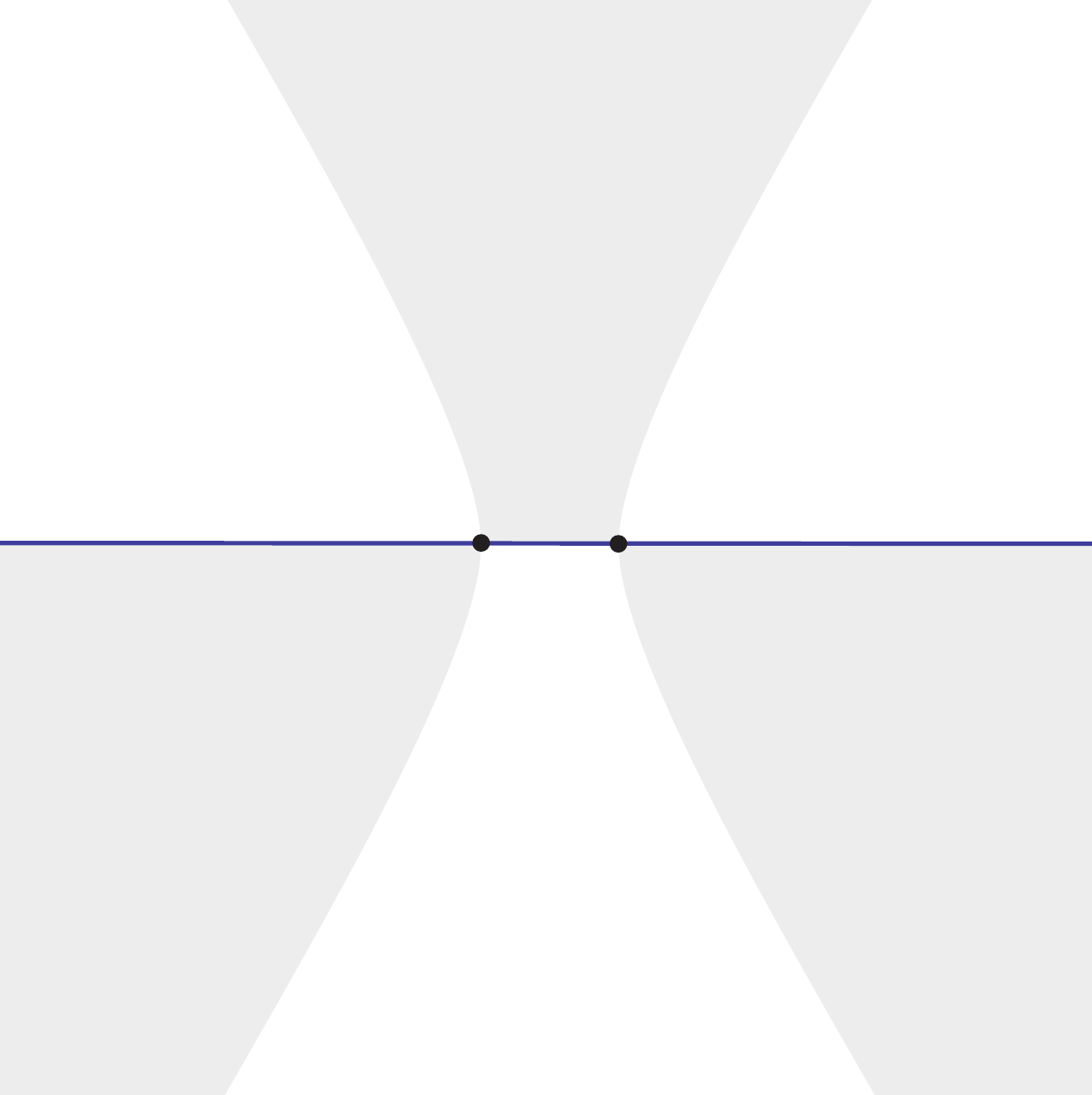}
      \put(102,48){\small $\R$}
      \put(55,44){\small $k_0$}
      \put(37,44){\small $-k_0$}
      \put(38,82){\small $\re \Phi > 0$}
      \put(70,61){\small $\re \Phi < 0$}
      \put(70,34){\small $\re \Phi > 0$}
      \put(38,13){\small $\re \Phi < 0$}    \end{overpic}
    \vspace{.5cm}
     \begin{figuretext}\label{criticalpointsfig}
        The critical points $\pm k_0$ in the complex $k$-plane in the case of $\zeta > 0$ (left) and $\zeta < 0$ (right). The regions where $\re \Phi > 0$ and $\re \Phi < 0$ are shaded and white, respectively. 
     \end{figuretext}
     \end{center}
\end{figure}

The proofs of our two theorems rely on a Deift-Zhou \cite{DZ1993, DZ1994} steepest descent analysis of the RH problem (\ref{preRHm}). 
There are two critical points (i.e., solutions of the equation $\partial_k\Phi = 0$) which are relevant for this analysis; they are given by $\pm k_0$, where (see Figure \ref{criticalpointsfig})
$$k_0 := \begin{cases}  i\sqrt{\frac{\zeta}{12}}, & \zeta \geq 0, \\
\sqrt{\frac{|\zeta|}{12}}, & \zeta < 0.
\end{cases}$$
%, we can obtain detailed asymptotic formulas for $m(x,t,k)$, and hence also for $u(x,t)$, as $(x,t)$ tends to infinity.

\subsection{Notation}\label{notationsubsec}
We will derive the asymptotics in the two halves of Sector IV corresponding to $x \leq 0$ and $x \geq 0$ separately; we use the notation 
$$\IV_\leq := \IV \cap \{x \leq 0\} \quad  \text{and} \quad \IV_\geq := \IV \cap \{x \geq 0\}.$$ 
If $D$ is an open connected subset of $\C$ bounded by a piecewise smooth curve $\partial D \subset \hat{\C} := \C \cup \{\infty\}$, then we write $\dot{E}^2(D)$ for the space of all analytic functions $f:D \to \C$ with the property that there exist curves $\{C_n\}_1^\infty$ in $D$ tending to $\partial D$ in the sense that $C_n$ eventually surrounds each compact subset of $D$ and such that
$$\sup_{n \geq 1} \int_{C_n} |f(z)|^2 |dz| < \infty.$$
If $D = D_1 \cup \cdots \cup D_n$ is a finite union of such open subsets, then $\dot{E}^2(D)$ denotes the space of functions $f:D\to \C$ such that $f|_{D_j} \in \dot{E}^2(D_j)$ for each $j$.
We can formulate the classical RH problem (\ref{preRHm}) in this setting as
\begin{align}\label{RHm}
\begin{cases}
m(x, t, \cdot) \in I + \dot{E}^2(\C \setminus \R),\\
m_+(x,t,k) = m_-(x, t, k) v(x, t, k) \quad \text{for a.e.} \ k \in \R.
\end{cases}
\end{align}
We write $E^\infty(D)$ for the space of bounded analytic functions on $D$.
We use $c>0$ and $C > 0$ to denote generic constants which may change within a computation. 

Given a piecewise smooth contour $\Gamma \subset \hat{\C}$, we denote the Cauchy operator $\mathcal{C} \equiv \mathcal{C}^\Gamma$ associated with $\Gamma$ by
\begin{align}\label{Cauchytransform}
(\mathcal{C}h)(k) = \frac{1}{2\pi i} \int_{\Gamma} \frac{h(k')dk'}{k' - k}, \qquad k \in \C \setminus \Gamma.
\end{align}
If $h \in L^2(\Gamma)$, then the left and right nontangential boundary values of $\mathcal{C} h$, which we denote by $\mathcal{C}_+ h$ and $\mathcal{C}_- h$ respectively, exist a.e. on $\Gamma$ and belong to $L^2(\Gamma)$. We let $\mathcal{B}(L^2(\Gamma))$ denote the space of bounded linear operators on $L^2(\Gamma)$ and, for $w \in L^\infty(\Gamma)$, we define $\mathcal{C}_{w} \in \mathcal{B}(L^2(\Gamma))$ by $\mathcal{C}_{w}h = \mathcal{C}_-(hw)$. 
We write $f^*(k) := \overline{f(\bar{k})}$ for the Schwartz conjugate of a function $f(k)$. 
If $A$ is an $n \times m$ matrix, we define $|A|\geq 0$ by
$|A|^2 = \sum_{i,j} |A_{ij}|^2$; then $|A + B| \leq |A| + |B|$ and $|AB| \leq |A| |B|$.
%$$im_x + u\sigma_3\sigma_1m = k[\sigma_3, m]$$
%i.e.
%$$m_x + ik[\sigma_3, m] = i\begin{pmatrix} 0 & u \\-u & 0 \end{pmatrix}m$$

%For the purpose of the proof of Theorem \ref{mainth1}, w

%the leading behavior of the solution $u(x,t)$ is given by 
%$$u(x,t) = \sqrt{\frac{\nu}{3tk_0}} \cos\big(16tk_0^3 - \nu \ln(192 t k_0^3) + \phi(k_0)) + \cdots,$$

\section{Asymptotics in Sector $\IV_\geq$}\label{sectorIVgeqsec}
We first consider the asymptotics in Sector $\IV_\geq$ given by
\begin{align*}
\IV_\geq = \{(x,t) \, | \, t\geq 1 \; \text{and} \; 0 \leq x \leq M t^{1/3}\},
\end{align*}
where $M > 0$ is a constant. In this sector, the critical points $\pm k_0 = \pm i \sqrt{\frac{x}{12 t}}$ are pure imaginary and approach $0$ at least as fast as $t^{-1/3}$ as $t \to \infty$, i.e., $|k_0| \leq C t^{-1/3}$. 

\subsection{The solution $m^{(1)}$}
The first step is to transform the RH problem in such a way that the new jump matrix approaches the identity matrix as $t \to \infty$. Let $\Gamma^{(1)} \subset \C$ denote the contour 
$$\Gamma^{(1)} := \R \cup e^{\frac{\pi i}{6}}\R \cup e^{-\frac{\pi i}{6}}\R$$ 
oriented to the right as in Figure \ref{Gamma1geq.pdf} and let $V$ and $V^*$ denote the open subsets shown in the same figure:
\begin{align*}
& V = \{\arg k \in (0, \pi/6)\} \cup \{\arg k \in (5\pi/6, \pi)\},
	\\
& V^* = \{\arg k \in (-\pi, -5\pi/6)\} \cup \{\arg k \in (-\pi/6, 0)\}.
\end{align*}

We first need to decompose $r$ into an analytic part $r_a$ and a small remainder $r_r$.
Let $N \geq 2$ be an integer. Assume for definiteness that $N$ is even.

\begin{figure}
\begin{center}
\begin{overpic}[width=.5\textwidth]{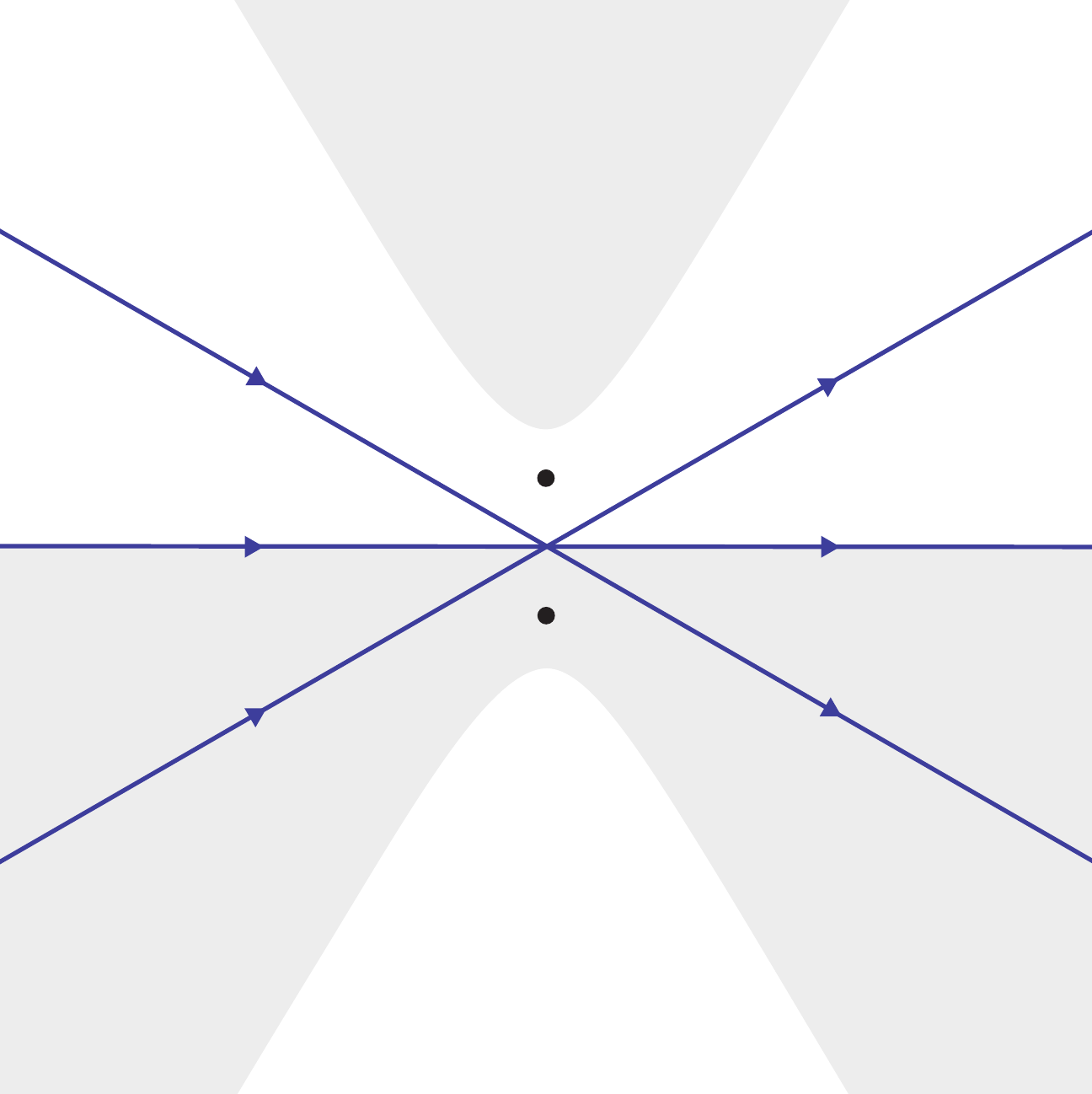}
      \put(103,48.5){\small $\Gamma^{(1)}$}
 %     \put(44,55){\small $k_0$}
%      \put(43,40){\small $-k_0$}
      \put(73,66.5){\small $1$}
      \put(24,66.5){\small $1$}
      \put(73,30.5){\small $2$}
      \put(24,30.5){\small $2$}
      \put(75,45){\small $3$}
      \put(22,45){\small $3$}
      \put(12,57){\small $V$}
      \put(12,39){\small $V^*$}
      \put(85,57){\small $V$}
      \put(85,39){\small $V^*$}
      \put(41,88){\small $\re \Phi > 0$}
    \end{overpic}
    \vspace{.5cm}
     \begin{figuretext}\label{Gamma1geq.pdf}
        The contour  $\Gamma^{(1)}$ and the sets $V$ and $V^*$ in the case of Sector $\IV_\geq$. The region where $\re \Phi > 0$ is shaded. 
     \end{figuretext}
     \end{center}
\end{figure}

\begin{lemma}[Analytic approximation for $\zeta \geq 0$]\label{decompositionlemmageq}
There exists a decomposition
\begin{align*}
& r(k) = r_{a}(t, k) + r_{r}(t, k), \qquad t \geq 1, \ k \in \R,
\end{align*}
where the functions $r_{a}$ and $r_{r}$ have the following properties:
\begin{enumerate}[$(a)$]
\item For each $t\geq1$, $r_{a}(t, k)$ is defined and continuous for $k \in \bar{V}$ and analytic for $k \in V$, where $\bar{V}$ denotes the closure of $V$.

\item The function $r_{a}$ obeys the following estimates uniformly for $\zeta \geq 0$ and $t \geq 1$:
\begin{align*}
& |r_{a}(t, k)| \leq \frac{C}{1 + |k|} e^{\frac{t}{4}|\re \Phi(\zeta,k)|}, \qquad
  k \in \bar{V},
\end{align*}
and
\begin{align}\label{raat1IV}
\bigg|r_{a}(t, k) - \sum_{j=0}^N \frac{r^{(j)}(0)}{j!} k^j\bigg| \leq C |k|^{N+1} e^{\frac{t}{4}|\re \Phi(\zeta,k)|}, \qquad k \in \bar{V}.
\end{align}

\item The $L^1$ and $L^\infty$ norms of $r_{r}(t, \cdot)$ on $\R$ are $O(t^{-N})$ as $t \to \infty$.

\item $r_{a}(t, k) = -r_{a}^*(t, -k)$ for $k \in \bar{V}$ and $r_{r}(t, k) = -r_{r}^*(t, -k)$ for $k \in \R$.

\end{enumerate}
\end{lemma}
\begin{proof}
The symmetry (\ref{rsymm}) implies that $\re r(k) = -\re r(-k)$ and $\im r(k) = \im r(-k)$ for $k \in \R$. It follows that
\begin{align}\label{rer3imr3}
(k^2 + 1)^{3N+4}k^{-1} \re r(k) \quad \text{and} \quad (k^2 + 1)^{3N+4}\im r(k)
\end{align}
are smooth even functions of $k \in \R$. Considering their Taylor expansions at $k = 0$, we infer the existence of real coefficients $\{a_j, b_j\}_{j=0}^{2N+1}$ such that
$$\begin{cases}
(k^2 + 1)^{3N+4}k^{-1}\re r(k) = \sum_{j = 0}^{2N+1} a_j k^{2j} + O(k^{4N+4}), \\
(k^2 + 1)^{3N+4}\im r(k) = \sum_{j = 0}^{2N+1} b_j k^{2j} + O(k^{4N+4}),
\end{cases} \ k \to 0,
$$
i.e.,
$$r(k) = r_0(k) + O\big(k^{4N+4}\big), \qquad k \to 0,$$
where the rational function $r_0(k)$ is defined by
$$r_0(k) = \frac{1}{(k^2+1)^{3N+4}} \sum_{j = 0}^{2N+1} (ka_j + ib_j) k^{2j}.$$
Note that $r_0(k)$ is analytic away from $k = \pm i$ and satisfies $r_0(k) = -r_0^*(-k)$. 
We conclude that $f(k) := r(k) - r_0(k)$ also obeys the symmetry  $f(k) = -f^*(-k)$, $k \in \R$, and that
\begin{align}\label{fcoincideIV}
 \frac{d^n f}{dk^n} (k) =
\begin{cases}
O(k^{4N+4-n}), \quad& k \to 0, 
	\\
O(k^{-2N-5}), &  k \to \pm \infty, 
 \end{cases}
 \  n = 0,1,\dots, N+1.
\end{align}

The decomposition of $r(k)$ can now be derived as follows.
The map $k \mapsto \phi \equiv \phi(k)$ defined by
\begin{align*}
  \phi = -i\Phi(0, k) = 8k^3
\end{align*}  
is an increasing bijection  $\R \to \R$, so we may define a function $F(\phi)$ by
\begin{align}\label{Fdef2I}
F(\phi) = \frac{(k^2+1)^{N+2}}{k^{N+1}} f(k), \qquad \phi \in \R.
\end{align}
The function $F(\phi)$ is smooth for $\phi \in \R \setminus \{0\}$ and 
\begin{align*}
\frac{d^n F}{d \phi^n}(\phi) = \bigg(\frac{1}{24k^2} \frac{d}{d k}\bigg)^n 
\bigg[\frac{(k^2+1)^{N+2}}{k^{N+1}} f(k)\bigg], \qquad \phi \in \R \setminus \{0\}.
\end{align*}
In view of (\ref{fcoincideIV}), we have $F \in C^N(\R)$, and, for $n = 0, 1, \dots, N+1$,
$$\bigg\|\frac{d^n F}{d \phi^n}\bigg\|_{L^2(\R)}^2
= \int_\R \bigg|\frac{d^n F}{d \phi^n}(\phi)\bigg|^2 d\phi
= \int_\R \bigg|\frac{d^n F}{d \phi^n}(\phi)\bigg|^2 \frac{d \phi}{d k} dk < \infty.$$
It follows that $F$ belongs to the Sobolev space $H^{N+1}(\R)$ and the Fourier transform $\hat{F}(s)$ defined by
\begin{align}\label{FhatdefIV}
\hat{F}(s) = \frac{1}{2\pi} \int_{\R} F(\phi) e^{-i\phi s} d\phi,
\end{align}
satisfies 
\begin{align}\label{FFhatIV}
F(\phi) =  \int_{\R} \hat{F}(s) e^{i\phi s} ds
\end{align}
and, by the Plancherel theorem, $\|s^{N+1} \hat{F}\|_{L^2(\R)} < \infty$.
By (\ref{Fdef2I}) and (\ref{FFhatIV}), we have
$$\frac{k^{N+1}}{(k^2+1)^{N+2}} \int_{\R} \hat{F}(s) e^{s\Phi(0,k)} ds 
= f(k), \qquad  k \in \R.$$
Let us write 
$$f(k) = f_{a}(t, k) + f_{r}(t, k), \qquad t\geq1, \ k \in \R,$$
where the functions $f_a$ and $f_r$ are defined by
\begin{align*}
& f_a(t,k) = \frac{k^{N+1}}{(k^2+1)^{N+2}}\int_{-\frac{t}{4}}^{\infty} \hat{F}(s) e^{s\Phi(0,k)} ds, \qquad   t\geq1, \  k \in \bar{V},  
	\\
& f_r(t,k) = \frac{k^{N+1}}{(k^2+1)^{N+2}}\int_{-\infty}^{-\frac{t}{4}} \hat{F}(s) e^{s\Phi(0,k)} ds,\qquad  t\geq1, \   k \in \R.
\end{align*}
Since 
$$\re \Phi(\zeta, k) \leq \re \Phi(0,k) \leq 0$$ 
for $\zeta \geq 0$ and $k \in \bar{V}$, the function $f_a(t, \cdot)$ is continuous in $\bar{V}$ and analytic in $V$, and we find
\begin{align*}\nonumber
 |f_a(t, k)| 
&\leq \frac{|k|^{N+1}}{|k^2+1|^{N+2}} \|\hat{F}\|_{L^1(\R)}  \sup_{s \geq -\frac{t}{4}} e^{s \re \Phi(0,k)}
\leq \frac{C|k|^{N+1}}{1 + |k|^{2N+4}}  e^{\frac{t}{4} |\re \Phi(0,k)|} 
	\\ 
&\leq \frac{C|k|^{N+1}}{1 + |k|^{2N+4}} e^{\frac{t}{4} |\re \Phi(\zeta,k)|}, \qquad \zeta \geq 0, \  t \geq 1, \  k \in \bar{V},
\end{align*}
and, by the Cauchy-Schwartz inequality,
\begin{align*}\nonumber
|f_r(t, k)| & \leq \frac{|k|^{N+1}}{|k^2+1|^{N+2}}  \int_{-\infty}^{-\frac{t}{4}} s^{N+1} |\hat{F}(s)| s^{-N-1} ds
	\\\nonumber
& \leq \frac{C}{1 + |k|^2}  \| s^{N+1} \hat{F}(s)\|_{L^2(\R)} \sqrt{\int_{-\infty}^{-\frac{t}{4}} |s|^{-2N-2} ds}  
 	\\ \nonumber
&  \leq \frac{C}{1 + |k|^2} t^{-N - \frac{1}{2}}, \qquad \zeta \geq 0, \  t \geq 1, \  k \in \R.
\end{align*}
In particulary, the $L^1$ and $L^\infty$ norms of $f_r$ on $\R$ are $O(t^{-N - \frac{1}{2}})$. 
The symmetries $f(k) = -f^*(-k)$ and $\phi(k) = -\phi(-k)$ together with the fact that $N$ is even imply that $F(\phi) = \overline{F(-\phi)}$. Hence $\hat{F}(s)$ is real valued for $s \in \R$. Since $\Phi(\zeta, k) = \Phi^*(\zeta, -k)$, this leads to the symmetries $f_a(t, k) = -f_a^*(t, -k)$ and $f_r(t, k) = -f_r^*(t, -k)$.
Letting
\begin{align*}
& r_{a}(t, k) = r_0(k) + f_a(t, k), \qquad t \geq 1, \ k \in \bar{V},
	\\
& r_{r}(t, k) = f_r(t, k), \qquad t \geq 1, \ k \in \R,
\end{align*}
we obtain a decomposition of $r$ with the asserted properties.
\end{proof}

Let $r = r_a + r_r$ be a decomposition of $r$ as in Lemma \ref{decompositionlemmageq} and define the sectionally analytic function $m^{(1)}$ by 
\begin{align}\label{m1def}
m^{(1)}(x,t,k) = m(x,t,k)G(x,t,k),
\end{align}
where
\begin{align}\label{GdefII}
G(x,t,k) = \begin{cases} 
\begin{pmatrix}  
 1 & 0 \\
 -r_{a} e^{t\Phi}  & 1
\end{pmatrix}, & k \in V,
	\\
\begin{pmatrix}  
 1 & - r_{a}^* e^{-t\Phi}  \\
0 & 1
\end{pmatrix}, & k \in V^*,
	\\
I, & \text{elsewhere}.
\end{cases}
\end{align}
Lemma \ref{decompositionlemmageq} implies that
$$G(x,t,\cdot) \in I + (\dot{E}^2 \cap E^\infty)(V \cup V^*).$$
Thus $m$ satisfies the RH problem (\ref{RHm}) if and only if $m^{(1)}$ satisfies the RH problem 
\begin{align}\label{RHm1}
\begin{cases}
m^{(1)}(x, t, \cdot) \in I + \dot{E}^2(\C \setminus \Gamma^{(1)}),\\
m^{(1)}_+(x,t,k) = m^{(1)}_-(x, t, k) v^{(1)}(x, t, k) \quad \text{for a.e.} \ k \in \Gamma^{(1)},
\end{cases}
\end{align}
where 
\begin{align}\nonumber
& v_1^{(1)} = \begin{pmatrix}  
 1 & 0 \\
 r_{a} e^{t\Phi}  & 1
\end{pmatrix}, \qquad
 v_2^{(1)} = \begin{pmatrix}  
 1 & - r_{a}^* e^{-t\Phi}  \\
 0  & 1
\end{pmatrix}, 
	\\\label{v1def}
&  v_3^{(1)} = \begin{pmatrix}  
 1 - |r_{r}|^2 & - r_{r}^* e^{-t\Phi}  \\
 r_{r} e^{t\Phi}  & 1
\end{pmatrix},
\end{align}
and $v_j^{(1)}$ denotes the restriction of $v^{(1)}$ to the subcontour labeled by $j$ in Figure \ref{Gamma1geq.pdf}.

\subsection{Local model}
The RH problem for $m^{(1)}$ has the property that the matrix $v^{(1)} - I$ is uniformly small as $t \to \infty$ everywhere on $\Gamma^{(1)}$ except possibly near the origin. 
Hence we only have to consider a neighborhood of the origin when computing the long-time asymptotics of $m^{(1)}$. In this subsection, we find a local solution $m_0$ which approximates $m^{(1)}$ near $0$. 

We introduce the variables $y$ and $z$ by
\begin{align}\label{yzdefIV}
y := \frac{x}{(3t)^{1/3}}, \qquad z := (3t)^{1/3}k.
\end{align}
Note that $0 \leq y \leq C$ in Sector $\IV_\geq$. The definitions of $y$ and $z$ are chosen such that
\begin{align}\label{tPhiminusyzIV}
t\Phi(\zeta, k) = 2i\bigg(y z + \frac{4z^3}{3}\bigg).
\end{align}

Fix $\epsilon > 0$. Let $D_\epsilon(0) = \{k \in \C\, | \, |k| < \epsilon\}$ denote the open disk of radius $\epsilon$ centered at the origin. The map $k \mapsto z$ maps $D_\epsilon(0)$ onto the open disk $D_{(3t)^{1/3}\epsilon}(0)$ of radius $(3t)^{1/3}\epsilon$ in the complex $z$-plane.

Let $\mathcal{Y}^\epsilon = (\Gamma^{(1)} \cap D_\epsilon(0))\setminus \R$. The map $k \mapsto z$ takes $\mathcal{Y}^\epsilon$ onto $Y \cap \{|z| < (3t)^{1/3}\epsilon\}$, where $Y$ is the contour defined in (\ref{YdefIV}).
We write $\mathcal{Y}^\epsilon = \cup_{j=1}^4\mathcal{Y}_j^\epsilon$, where $\mathcal{Y}_j^\epsilon$ denotes the part of $\mathcal{Y}^\epsilon$ that maps into $Y_j$, see Figure \ref{calYgeq.pdf}.

\begin{figure}
\begin{center}
 \begin{overpic}[width=.4\textwidth]{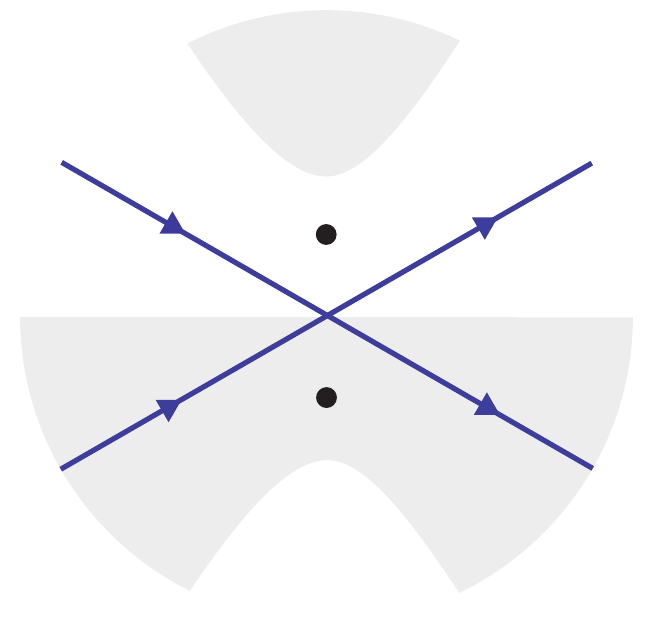}
 \put(71,69){\small $\mathcal{Y}_1^\epsilon$}
 \put(24,69){\small $\mathcal{Y}_2^\epsilon$}
 \put(24,28){\small $\mathcal{Y}_3^\epsilon$}
 \put(71,28){\small $\mathcal{Y}_4^\epsilon$}
 \put(48.5,66){\small $k_0$}
 \put(45,30){\small $-k_0$}
 \end{overpic}
   \begin{figuretext}\label{calYgeq.pdf}
      The contour $\mathcal{Y}^\epsilon = \cup_{j=1}^4\mathcal{Y}_j^\epsilon$.
      \end{figuretext}
   \end{center}
\end{figure}

The long-time asymptotics in Sector $\IV_\geq$ is related to the solution $m^Y(y,t,z)$ of the RH problem (\ref{RHmYIV}) with the polynomial $p$ in (\ref{psumIV}) given by 
\begin{align}\label{pNdefIV}
  p(t,z) = \sum_{j=0}^N \frac{r^{(j)}(0)}{j!} k^j = \sum_{j=0}^N \frac{r^{(j)}(0)}{j!3^{j/3}} \frac{z^j}{t^{j/3}}.
\end{align}
In particular, the first few coefficients in (\ref{psumIV}) are 
\begin{align}\label{sp1p2}
s = r(0) \in i\R, \qquad p_1 = \frac{r'(0)}{3^{1/3}} \in \R, \qquad p_2 = \frac{r''(0)}{2 \times 3^{2/3}} \in i\R.
\end{align}
Define $m_0(x,t,k)$ for $k \in D_\epsilon(0)$ by
\begin{align}\label{m0defIV}
m_0(x, t, k) = m^Y(y, t, z), \qquad k \in D_\epsilon(0).
\end{align}
By Lemma \ref{YlemmaIV}, we can choose $T \geq 1$ such that $m_0$ is well-defined whenever $(x,t) \in \IV_\geq^T$, where $\IV_\geq^T$ denotes the part of $\IV_\geq$ where $t\geq T$, i.e.,
$$\IV_\geq^T := \IV_\geq \cap \{t \geq T\}.$$ 
By (\ref{rsymm}), we have $p(t,z) = -\overline{p(t,-\bar{z})}$. Hence (\ref{mYsymmIV}) implies that $m_0$ obeys the same symmetries (\ref{msymm}) as $m$.

\begin{lemma}\label{m0lemmaIV}
For each $(x,t) \in \IV_\geq^T$, the function $m_0(x,t,k)$ defined in (\ref{m0defIV}) is an analytic function of $k \in D_\epsilon(0) \setminus \mathcal{Y}^\epsilon$ such that
\begin{align}\label{m0boundIV}
|m_0(x,t,k) | \leq C, \qquad (x,t) \in \IV_\geq^T, \ k \in D_\epsilon(0) \setminus \mathcal{Y}^\epsilon.
\end{align}
Across $\mathcal{Y}^\epsilon$, $m_0$ obeys the jump condition $m_{0+} =  m_{0-} v_0$, where the jump matrix $v_0$ satisfies, for each $1 \leq p \leq \infty$,
\begin{align}\label{v2v0estimateIV}
 \|v^{(1)} - v_0\|_{L^p(\mathcal{Y}^\epsilon)} \leq Ct^{-\frac{N+1}{3}}, \qquad (x,t) \in \IV_\geq^T.
\end{align}	
Furthermore, 
\begin{align}\label{m0invexpansionIV}
 m_0(x,t,k)^{-1}  = I +  \sum_{j=1}^N \sum_{l=0}^N \frac{m_{0,jl}(y)}{k^j t^{(j+l)/3}} +  O\big(t^{-\frac{N+1}{3}}\big)
\end{align}  
uniformly for $(x,t) \in \IV_\geq^T$ and $k \in \partial D_\epsilon(0)$, where the coefficients $m_{0,jl}(y)$ are smooth functions of $y \in [0,\infty)$.
In particular, 
\begin{align}\label{m0LinftyestimateIV}
\|m_0(x,t,\cdot)^{-1} - I\|_{L^\infty(\partial D_\epsilon(0))} = O\big(t^{-1/3}\big), 
\end{align}
and 
\begin{align}\label{m0circleIV}
& \frac{1}{2\pi i}\int_{\partial D_\epsilon(0)}(m_0(x,t,k)^{-1} - I) dk
= \sum_{l=1}^N \frac{g_l(y)}{t^{l/3}} + O\big(t^{-\frac{N+1}{3}}\big),
\end{align}	
uniformly for $(x,t) \in \IV_\geq^T$, where $g_{l+1}(y) = m_{0,1l}(y)$ for each $l \geq 0$; the first coefficient is given by
\begin{align}\label{g1explicitIV}
g_1(y) =  - 
\frac{3^{-1/3}}{2} \begin{pmatrix} -i\int_y^\infty u_P^2(y'; s, 0, -s)dy' & u_P(y; s, 0, -s) \\ u_P(y; s, 0, -s) & i\int_y^\infty u_P^2(y'; s, 0, -s)dy'  \end{pmatrix}
\end{align}
with $s = r(0)$. If $r(0) = 0$, then
\begin{align}\nonumber
g_1(y) = &\; 0,
	\\ \nonumber
g_2(y) = &\; -\frac{r'(0)}{4\times 3^{2/3}} \Ai'(y)\sigma_1
	\\ \label{g1g2g3explicit}
g_3(y) = &\; \frac{i r'(0)^2}{24} \int_y^{\infty} (\Ai'(y'))^2dy' \sigma_3
 +  \frac{ir''(0)}{48} y\Ai(y) \sigma_1.
\end{align}
\end{lemma}
\begin{proof}
The analyticity of $m_0$ and the bound (\ref{m0boundIV}) are a consequence of Lemma \ref{YlemmaIV}.
Moreover,
\begin{align}\label{v2minusv0IV}
& v^{(1)} - v_0 = \begin{cases}
   \begin{pmatrix}  
 0 & 0  \\
 (r_a(t,k) - p(t,z)) e^{t\Phi}   & 0 
  \end{pmatrix}, \qquad k \in \mathcal{Y}_1^\epsilon \cup \mathcal{Y}_2^\epsilon,
  	\\ 
 \begin{pmatrix}  
 0 & -(r_a^*(t,k) - p^*(t, z)) e^{-t\Phi} \\
0 & 0 
  \end{pmatrix}, \qquad k \in \mathcal{Y}_3^\epsilon \cup \mathcal{Y}_4^\epsilon.
\end{cases}
\end{align}
Since
$$\re \Phi(\zeta, re^{\frac{\pi i}{6}}) = -8 r^3 - \zeta r \leq -8r^3$$
for $\zeta \geq 0$ and $r \geq 0$, we obtain 
\begin{align}\label{PhionY1IV}
\re \Phi(\zeta, k) \leq -8|k|^3, \qquad k \in \mathcal{Y}_1^\epsilon, \ (x,t) \in \IV_\geq^T.
\end{align}
In particular,
\begin{align*}
e^{-\frac{3t}{4} |\re \Phi|}  \leq Ce^{-6t|k|^3} \leq C e^{-2|z|^3}, \qquad k \in \mathcal{Y}_1^\epsilon.
\end{align*}
Thus, by (\ref{v2minusv0IV}), (\ref{raat1IV}), and (\ref{pNdefIV}), 
\begin{align*}\nonumber
|v^{(1)} - v_0|
& \leq C|r_a(t,k) - p(t,z)| e^{t \re \Phi}
\leq C|zt^{-1/3}|^{N+1} e^{-\frac{3t}{4} |\re \Phi|} 
	\\
& \leq C|zt^{-1/3}|^{N+1} e^{-2|z|^3}, \qquad k \in \mathcal{Y}_1^\epsilon.
\end{align*}
Consequently, writing $r = |z|$,
\begin{align*}
& \|v^{(1)} - v_0\|_{L^\infty(\mathcal{Y}_1^{\epsilon})}
 \leq C\sup_{0 \leq r < \infty} (rt^{-1/3})^{N+1} e^{-2r^3}
\leq C t^{-(N+1)/3}
\end{align*}
and
\begin{align*}
& \|v^{(1)} - v_0\|_{L^1(\mathcal{Y}_1^{\epsilon})} 
 \leq C\int_0^{\infty} (rt^{-1/3})^{N+1} e^{-2r^3} \frac{dr}{t^{1/3}}
\leq C t^{-(N+2)/3}.
\end{align*}
Since similar estimates apply to $\mathcal{Y}_j^{\epsilon}$ with $j = 2,3,4$, this proves (\ref{v2v0estimateIV}).

We next apply Lemma \ref{YlemmaIV} to determine the asymptotics of the solution $m_0$. 
The variable $z = (3t)^{1/3}k$ satisfies $|z| = (3t)^{1/3}\epsilon$ if $|k| = \epsilon$. 
Thus equation (\ref{mYasymptoticsIV}) yields
\begin{align}\label{m0expansionmYIV}
  m_0(x,t,k) = &\; I + \sum_{j=1}^N \sum_{l=0}^N \frac{m_{jl}^Y(y)}{((3t)^{1/3}k)^j t^{l/3}}
 + O\big(t^{-\frac{N+1}{3}}\big)
\end{align}  
uniformly for $(x,t) \in \IV_\geq^T$ and $k \in \partial D_\epsilon(0)$.
It follows that the expansion (\ref{m0invexpansionIV}) exists with coefficients $m_{0,jl}(y)$ which can be expressed in terms of the $m_{jl}^Y(y)$; the first coefficient is given by 
$m_{0,10}(y) = - 3^{-1/3}m_{10}^Y(y)$.
Equation (\ref{m0invexpansionIV}) and Cauchy's formula yield
\begin{align}\label{intDepsm0inviV}
\frac{1}{2\pi i}\int_{\partial D_\epsilon(0)}(m_0^{-1} - I) dk
= &\; \sum_{l=0}^{N} \frac{m_{0,1l}(y)}{t^{\frac{l+1}{3}}} 
+ O\big(t^{-\frac{N+2}{3}}\big).
\end{align}
This proves (\ref{m0circleIV}). The expression (\ref{g1explicitIV}) follows from (\ref{m10Yexplicit}) because 
$$g_1(y) = m_{0,10}(y) = - 3^{-1/3}m_{10}^Y(y).$$
The estimate (\ref{m0LinftyestimateIV}) follows from (\ref{m0invexpansionIV}).

If $r(0) = 0$, then $m_{10}^Y(y) = 0$ and the explicit expressions in (\ref{g1g2g3explicit}) then follow from (\ref{sp1p2}), (\ref{m0invexpansionIV}), (\ref{m0expansionmYIV}), and (\ref{intDepsm0inviV}) by straightforward computations.
\end{proof}

\subsection{The solution $\hat{m}$} 
Let $\hat{\Gamma} := \Gamma^{(1)} \cup \partial D_\epsilon(0)$ and assume that the boundary of $D_\epsilon(0)$ is oriented counterclockwise, see Figure \ref{Gammahatgeq.pdf}. 
The function $\hat{m}(x,t,k)$ defined by 
\begin{align}\label{mhatdef}
\hat{m}(x,t,k) = \begin{cases}
m^{(1)}(x, t, k)m_0(x,t,k)^{-1}, & k \in D_\epsilon(0),\\
m^{(1)}(x, t, k), & k \in \C \setminus D_\epsilon(0),
\end{cases}
\end{align}
satisfies the RH problem 
\begin{align}\label{RHmhat}
\begin{cases}
\hat{m}(x, t, \cdot) \in I + \dot{E}^2(\C \setminus \hat{\Gamma}),\\
\hat{m}_+(x,t,k) = \hat{m}_-(x, t, k) \hat{v}(x, t, k) \quad \text{for a.e.} \ k \in \hat{\Gamma},
\end{cases}
\end{align}
where the jump matrix $\hat{v}(x,t,k)$ is given by 
\begin{align}\label{vhatdef}
\hat{v} 
=  \begin{cases}
 m_{0-} v^{(1)} m_{0+}^{-1}, & k \in \hat{\Gamma} \cap D_\epsilon(0), \\
m_0^{-1}, & k \in \partial D_\epsilon(0), \\
v^{(1)},  & k \in \hat{\Gamma} \setminus \overline{D_\epsilon(0)}.
\end{cases}
\end{align}
We write $\hat{\Gamma}$ as the union of four subcontours as follows:
$$\hat{\Gamma} = \partial D_\epsilon(0) \cup \mathcal{Y}^\epsilon \cup \R \cup \hat{\Gamma}',$$
where $\hat{\Gamma}' := \Gamma^{(1)} \setminus (\R \cup \overline{D_\epsilon(0)})$.

\begin{figure}
\begin{center}
\begin{overpic}[width=.5\textwidth]{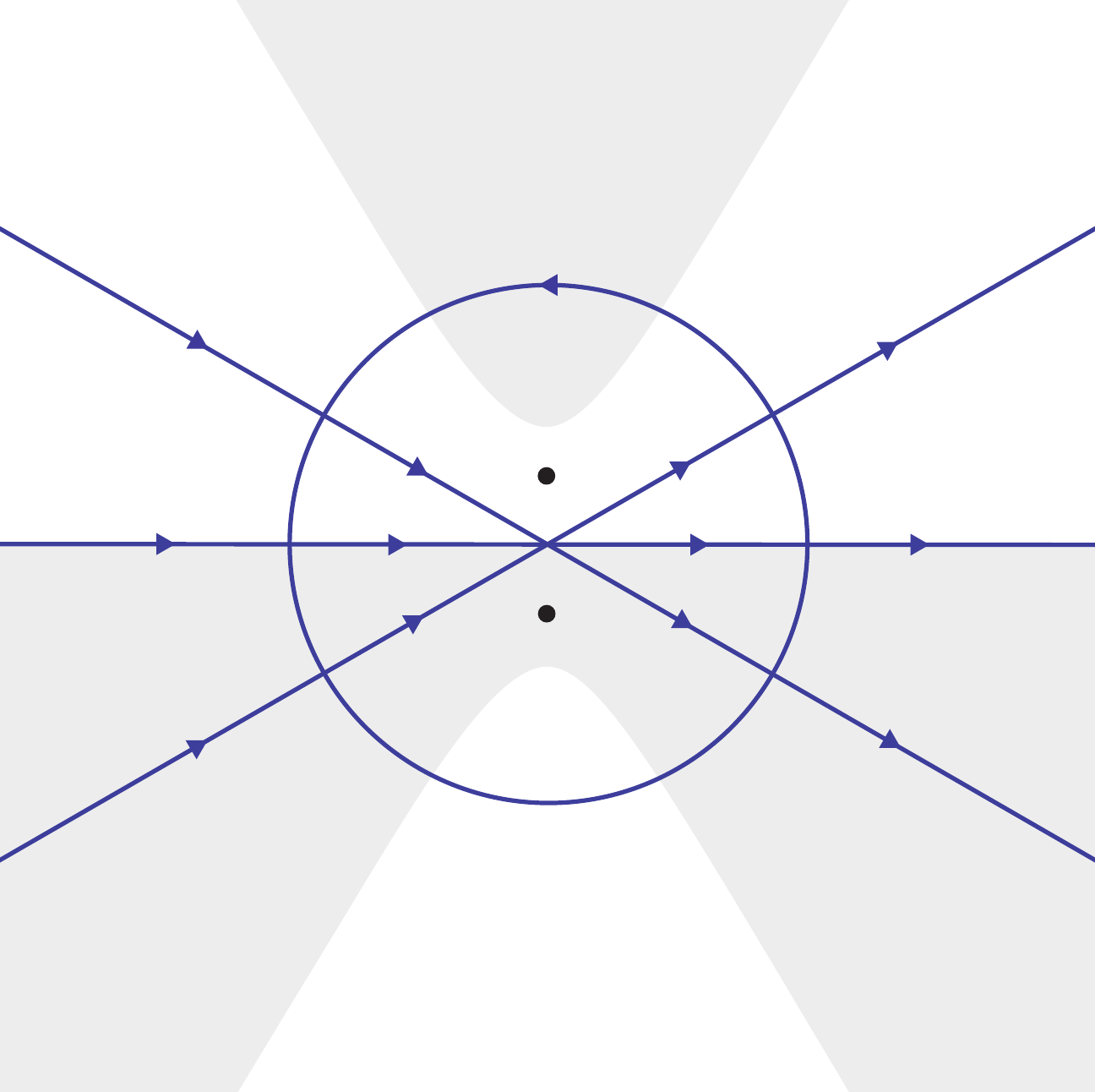}
      \put(103,48){\small $\hat{\Gamma}$}
      \put(75,46.5){\small $\epsilon$}
      \put(19.5,46.5){\small $-\epsilon$}
    %  \put(49,45){\small $0$}
    \end{overpic}
     \begin{figuretext}\label{Gammahatgeq.pdf}
        The contour  $\hat{\Gamma}$ in the case of Sector $\IV_\geq$.
     \end{figuretext}
     \end{center}
\end{figure}

\begin{lemma}\label{wlemmaIV}
Let $\hat{w} = \hat{v} - I$. For each $1 \leq p \leq \infty$, the following estimates hold uniformly for $(x,t) \in \IV_\geq^T$:
\begin{subequations}\label{westimateIV}
\begin{align}\label{westimateIVa}
& \|\hat{w}\|_{L^p(\partial D_\epsilon(0))} \leq Ct^{-1/3},
	\\\label{westimateIVb}
& \|\hat{w}\|_{L^p(\mathcal{Y}^\epsilon)} \leq Ct^{-(N+1)/3}, 
	\\ \label{westimateIVc}
& \|\hat{w}\|_{L^p(\R)} \leq C t^{-N},  
	\\\label{westimateIVd}
& \|\hat{w}\|_{L^p(\hat{\Gamma}')} \leq Ce^{-ct}.
\end{align}
\end{subequations}
\end{lemma}
\begin{proof}
The estimate (\ref{westimateIVa}) follows from (\ref{m0LinftyestimateIV}).
For $k \in \mathcal{Y}^\epsilon$, we have
$$\hat{w} = m_{0-} (v^{(1)} - v_0) m_{0+}^{-1},$$
so (\ref{m0boundIV}) and (\ref{v2v0estimateIV}) yield (\ref{westimateIVb}).
On $\R$, the jump matrix $v^{(1)}$ involves the small remainder $r_r$ (see the expression for $v_3^{(1)}$ in (\ref{v1def})), so the estimates (\ref{westimateIVc}) hold as a consequence of Lemma \ref{decompositionlemmageq} and (for the part of $\R$ that lies in $D_\epsilon(0)$) the boundedness (\ref{m0boundIV}) of $m_0$.
Finally, (\ref{westimateIVd}) follows because $e^{-t|\re \Phi|} \leq Ce^{-ct}$ uniformly on $\hat{\Gamma}'$. 
\end{proof}

The estimates in Lemma \ref{wlemmaIV} show that
\begin{align}\label{hatwLinftyIV}
\|\hat{w}\|_{(L^1 \cap L^\infty)(\hat{\Gamma})} \leq Ct^{-1/3},	 \qquad (x,t) \in \IV_\geq^T.
\end{align}
In particular, $\|\hat{w}\|_{L^\infty(\hat{\Gamma})} \to 0$ uniformly as $t \to \infty$. Thus, increasing $T$ if necessary, we may assume that 
$$\|\hat{\mathcal{C}}_{\hat{w}}\|_{\mathcal{B}(L^2(\hat{\Gamma}))} \leq C \|\hat{w}\|_{L^\infty(\hat{\Gamma})} \leq 1/2$$
for all $(x,t) \in \IV_\geq^T$, where $\hat{\mathcal{C}}_{\hat{w}}f := \hat{\mathcal{C}}_-(f\hat{w})$ and $\hat{\mathcal{C}}$ is the Cauchy operator associated with $\hat{\Gamma}$, see Section \ref{notationsubsec}. It follows from a standard argument that if $(x,t) \in \IV_\geq^T$, then the RH problem (\ref{RHmhat}) for $\hat{m}$ has a unique solution given by 
\begin{align}\label{mhatrepresentation}
\hat{m}(x, t, k) = I + \hat{\mathcal{C}}(\hat{\mu} \hat{w}) = I + \frac{1}{2\pi i}\int_{\hat{\Gamma}} (\hat{\mu} \hat{w})(x, t, s) \frac{ds}{s - k},
\end{align}
where $\hat{\mu}(x, t, k) \in I + L^2(\hat{\Gamma})$ is defined by $\hat{\mu} = I + (I - \hat{\mathcal{C}}_{\hat{w}})^{-1}\hat{\mathcal{C}}_{\hat{w}}I$.
Also, by (\ref{hatwLinftyIV}),
\begin{align}\label{muestimateIV}
\|\hat{\mu}(x,t,\cdot) - I\|_{L^2(\hat{\Gamma})} 
\leq  \frac{C\|\hat{w}\|_{L^2(\hat{\Gamma})}}{1 - \|\hat{\mathcal{C}}_{\hat{w}}\|_{\mathcal{B}(L^2(\hat{\Gamma}))}}\leq C t^{-1/3}, \qquad (x,t) \in \IV_\geq^T.
\end{align}

By (\ref{m0invexpansionIV}) and (\ref{westimateIV}), $\hat{w}$ has an expansion to order $O(t^{-(N+1)/3})$ in $t$ as $t \to \infty$, i.e.,
$$\hat{w}(x, t,k) = \frac{\hat{w}_1(y,k)}{t^{1/3}}+ \frac{\hat{w}_2(y,k)}{t^{2/3}} + \cdots + \frac{\hat{w}_N(y,k)}{t^{N/3}} + \frac{\hat{w}_{err}(x,t,k)}{t^{(N+1)/3}}, \quad (x,t) \in \IV_\geq^T, \ k \in \hat{\Gamma},$$
where the coefficients $\{\hat{w}_j\}_1^N$ are nonzero only for $k \in \partial D_\epsilon(0)$ and, for $1 \leq p \leq \infty$ and $ j = 1, \dots, N$,
\begin{align}\label{hatwerrestIV}
\begin{cases}
\|\hat{w}_j(y,\cdot)\|_{L^p(\hat{\Gamma})}  \leq C,\\
\|\hat{w}_{err}(x,t,\cdot)\|_{L^p(\hat{\Gamma})}  \leq C,
\end{cases} \qquad (x,t) \in \IV_\geq^T.
\end{align}
Since $\hat{\mu} = \sum_{j=0}^N \hat{\mathcal{C}}_{\hat{w}}^jI + (I-\hat{\mathcal{C}}_{\hat{w}})^{-1}\hat{\mathcal{C}}_{\hat{w}}^{N+1}I$ and
$$\hat{\mathcal{C}}_{\hat{w}} = \frac{\hat{\mathcal{C}}_{\hat{w}_1}}{t^{1/3}} + \frac{\hat{\mathcal{C}}_{\hat{w}_2}}{t^{2/3}} + \cdots + \frac{\hat{\mathcal{C}}_{\hat{w}_N}}{t^{N/3}} + \frac{\hat{\mathcal{C}}_{\hat{w}_{err}}}{t^{(N+1)/3}},$$
it follows that (see the explanation of (\ref{muYexpansionIV}) for more details of a similar argument)
$$\hat{\mu}(x, t,k) = I + \frac{\hat{\mu}_1(y,k)}{t^{1/3}} + \cdots
+ \frac{\hat{\mu}_N(y,k)}{t^{N/3}} + \frac{\hat{\mu}_{err}(x,t,k)}{t^{(N+1)/3}}, \qquad (x,t) \in \IV_\geq^T, \ k \in \hat{\Gamma},$$
where the coefficients $\{\hat{\mu}_j(y,k)\}_1^N$ depend smoothly on $y \in [0,\infty)$ and
\begin{align}\label{hatmuerrestIV}
\begin{cases}
 \|\hat{\mu}_j(y,\cdot)\|_{L^2(\hat{\Gamma})} \leq C, &
	\\ 
\|\hat{\mu}_{err}(x,t,\cdot)\|_{L^2(\hat{\Gamma})} \leq C, &
\end{cases} \ (x,t) \in \IV_\geq^T, \  j = 1, \dots, N.
\end{align}

\subsection{Asymptotics of $u$}
By expanding (\ref{mhatrepresentation}) and inverting the transformations (\ref{m1def}) and (\ref{mhatdef}), we find the relation
\begin{align}\label{limkm}
\lim_{k\to \infty} k(m(x,t,k)  - I)
= -\frac{1}{2\pi i}\int_{\hat{\Gamma}} \hat{\mu}(x, t, k) \hat{w}(x, t, k) dk.
\end{align}
We will use (\ref{limkm}) to compute the large $t$ asymptotics of $\lim_{k\to \infty} k(m - I)$; then the asymptotics of $u(x,t)$ will follow from (\ref{recoveru}).
From now on until the end of Section \ref{sectorIVgeqsec}, all equations involving error terms of the form $O(\cdot)$ will be valid uniformly for all $(x,t) \in \IV_\geq^T$.

By (\ref{m0circleIV}), (\ref{hatwerrestIV}), and (\ref{hatmuerrestIV}), the contribution from $\partial D_\epsilon(0)$ to the right-hand side of (\ref{limkm}) is
\begin{align*}
& -\frac{1}{2\pi i} \int_{\partial D_\epsilon(0)} \hat{w} dk	
-  \frac{1}{2\pi i} \int_{\partial D_\epsilon(0)} (\hat{\mu} - I) \hat{w} dk
= -\frac{1}{2\pi i}\int_{\partial D_\epsilon(0)} (m_0^{-1} - I) dk
	\\
& -  \frac{1}{2\pi i} \int_{\partial D_\epsilon(0)} \bigg(\sum_{j=1}^N \frac{\hat{\mu}_j(y,k)}{t^{j/3}} + \frac{\hat{\mu}_{err}(x,t,k)}{t^{(N+1)/3}}\bigg) 
 \bigg(\sum_{j=1}^N \frac{\hat{w}_j(y,k)}{t^{j/3}} + \frac{\hat{w}_{err}(x,t,k)}{t^{(N+1)/3}}\bigg)  dk
	\\
& =  - \sum_{j=1}^{N} \frac{h_j(y)}{t^{j/3}} + O\big(t^{-\frac{N+1}{3}}\big),
\end{align*}
where $\{h_j(y)\}_1^N$ are smooth matrix-valued functions of $y \in [0,\infty)$ and $h_1(y) = g_1(y)$.
By (\ref{westimateIVb}) and (\ref{muestimateIV}), the contribution from $\mathcal{Y}^\epsilon$ to the right-hand side of (\ref{limkm}) is
\begin{align*}
 -\frac{1}{2\pi i}\int_{\mathcal{Y}^\epsilon} \hat{\mu} \hat{w} dk
= O\big(\|\hat{w}\|_{L^1(\mathcal{Y}^\epsilon)}
& + \|\hat{\mu} - I\|_{L^2(\mathcal{Y}^\epsilon)}\|\hat{w}\|_{L^2(\mathcal{Y}^\epsilon)}\big)
  = O(t^{-(N+1)/3}).
\end{align*}
Similarly, by (\ref{westimateIVc}) and (\ref{muestimateIV}), the contribution from $\R$ is $O(t^{-N})$, and, by (\ref{westimateIVd}) and (\ref{muestimateIV}), the contribution from $\hat{\Gamma}'$ is $O(e^{- c t})$.
In summary, we arrive at
\begin{align}\label{limkmfinalIV}
\lim_{k\to \infty} k(m(x,t,k)  - I)
= - \sum_{j=1}^{N} \frac{h_j(y)}{t^{j/3}} + O\big(t^{-\frac{N+1}{3}}\big).
\end{align}
Recalling (\ref{recoveru}), this leads to the asymptotic formula
\begin{align*}
u(x,t) & = 2 \lim_{k\to \infty}k(m(x,t,k))_{21}
= \sum_{j=1}^{N} \frac{u_j(y)}{t^{j/3}}+ O(t^{-(N+1)/3}),
\end{align*}
where the $u_j(y)$ are smooth functions of $y \in [0,\infty)$ given by
\begin{align}\label{ujexpressionIV}
u_j(y) = -2 (h_j(y))_{21}, \qquad j = 1, \dots, N.
\end{align}
Using the expression (\ref{g1explicitIV}) for $h_1(y) = g_1(y)$, we find that $u_1(y)$ is given by (\ref{u1expression}). This completes the proof of Theorem \ref{mainth1} for Sector $\IV_\geq$.

Let us finally assume that $r(0) = 0$.
Then $\hat{w}_1 = \hat{\mu}_1 = 0$, and so $h_j(y) = g_j(y)$ for $j = 1,2,3$.
Substituting the formulas (\ref{g1g2g3explicit}) for $\{g_j\}_1^3$ into (\ref{ujexpressionIV}) gives $u_1 = 0$ and the expressions in (\ref{u2u3explicit}) for $u_2(y)$ and $u_3(y)$, thus completing the proof in Sector $\IV_\geq$ also for Theorem \ref{mainth2}.

\section{Asymptotics in Sector $\IV_\leq$}\label{sectorIVleqsec}
We now consider the asymptotics in the sector $\IV_\leq$ defined by
\begin{align*}
\IV_\leq = \{(x,t) \, | \, t\geq 1 \; \text{and} \; -M t^{1/3} \leq x \leq 0\},
\end{align*}
where $M > 0$ is a constant. In this sector, the critical points $\pm k_0 = \pm \sqrt{\frac{|x|}{12 t}}$ are real and approach $0$ at least as fast as $t^{-1/3}$ as $t \to \infty$, i.e., $|k_0| \leq C t^{-1/3}$.

\subsection{The solution $m^{(1)}$}
As in Section \ref{sectorIVgeqsec}, we begin by decomposing $r$ into an analytic part $r_a$ and a small remainder $r_r$. This time we define the contour $\Gamma^{(1)} \equiv \Gamma^{(1)}(\zeta)$ and the open subsets $V \equiv V(\zeta)$ and $V^* \equiv V^*(\zeta)$ as in Figure \ref{Gamma1leq.pdf}.
We still assume that $N \geq 2$ is an integer. We let $A > 0$ be a constant.

\begin{figure}
\begin{center}
\begin{overpic}[width=.5\textwidth]{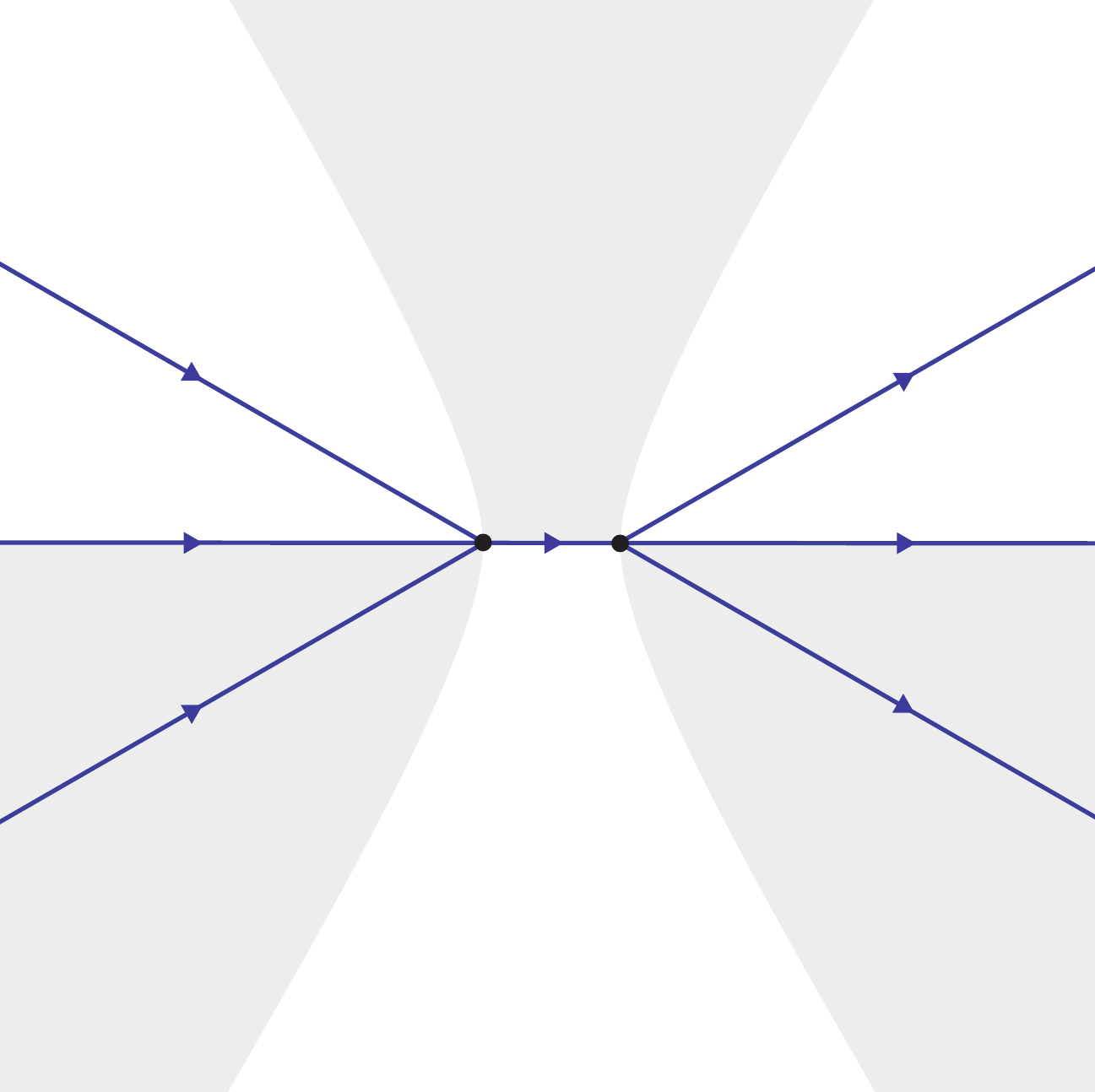}
      \put(103,48.5){\small $\Gamma^{(1)}$}
      \put(80,67){\small $1$}
      \put(18,67){\small $1$}
      \put(81,30){\small $2$}
      \put(18,30){\small $2$}
      \put(17,45.5){\small $3$}
      \put(81.5,45.5){\small $3$}
      \put(49,52){\small $4$}
      \put(55,45.5){\small $k_0$}
      \put(40,45.5){\small $-k_0$}
      \put(8,57){\small $V$}
      \put(8,39){\small $V^*$}
      \put(89,57){\small $V$}
      \put(89,39){\small $V^*$}
      \put(41,84){\small $\re \Phi > 0$}
    \end{overpic}
    \vspace{.5cm}
     \begin{figuretext}\label{Gamma1leq.pdf}
        The contour  $\Gamma^{(1)}$ and the sets $V$ and $V^*$ in the case of Sector $\IV_\leq$. The region where $\re \Phi > 0$ is shaded. 
     \end{figuretext}
     \end{center}
\end{figure}

\begin{lemma}[Analytic approximation for $-A \leq \zeta \leq 0$]\label{decompositionlemmaleq}
There exists a decomposition
\begin{align*}
& r(k) = r_{a}(x, t, k) + r_{r}(x, t, k), \qquad k \in (-\infty, -k_0)\cup (k_0, \infty),
\end{align*}
where the functions $r_{a}$ and $r_{r}$ have the following properties:
\begin{enumerate}[$(a)$]
\item For each $\zeta \in [-A, 0]$ and $t\geq1$, $r_{a}(x, t, k)$ is defined and continuous for $k \in \bar{V}$ and analytic for $k \in V$.

\item The function $r_{a}$ obeys the following estimates uniformly for $\zeta \in [-A, 0]$ and $t \geq 1$:
\begin{align*}
& |r_{a}(x, t, k)| \leq \frac{C}{1 + |k|} e^{\frac{t}{4}|\re \Phi(\zeta,k)|}, \qquad
  k \in \bar{V}, 
\end{align*}
and
\begin{align}\label{raatk0}
\bigg|r_{a}(x, t, k) - \sum_{j=0}^N \frac{r^{(j)}(k_0)}{j!} (k-k_0)^j\bigg| \leq C |k-k_0|^{N+1} e^{\frac{t}{4}|\re \Phi(\zeta,k)|}, \qquad k \in \bar{V}.
\end{align}

\item The $L^1$ and $L^\infty$ norms of $r_{r}(x, t, \cdot)$ on $(-\infty, -k_0)\cup (k_0, \infty)$ are $O(t^{-N})$ as $t \to \infty$ uniformly with respect to $\zeta \in [-A, 0]$.

\item $r_{a}(x, t, k) = -r_{a}^*(x, t, -k)$ for $k \in \bar{V}$ and $r_{r}(x, t, k) = -r_{r}^*(x, t, -k)$ for $k \in (-\infty, -k_0)\cup (k_0, \infty)$.

\end{enumerate}
\end{lemma}
\begin{proof}
We write $V = V' \cup V''$, where $V'$ and $V''$ denote the parts of $V$ in the right and left half-planes, respectively. We will derive a decomposition of $r$ in $V'$ and then use the symmetry (\ref{rsymm}) to extend it to $V''$. 

Taylor expansion of $\rho(k) := (k-i)^{4N+7} r(k)$ around $k_0$ gives
$$(k-i)^{4N+7} r(k) = \sum_{j=0}^{4N+3} \frac{\rho^{(j)}(k_0)}{j!}(k-k_0)^j + \frac{1}{(4N+3)!} \int_{k_0}^k \rho^{(4N+4)}(u) (k-u)^{4N+3} du.$$
Thus the function
$$f(\zeta,k) := r(k) - r_0(\zeta,k)
= \frac{1}{(4N+3)!(k-i)^{4N+7}}  \int_{k_0}^k \rho^{(4N+4)}(u) (k-u)^{4N+3} du,$$
where
$$r_0(\zeta, k) := \sum_{j=0}^{4N+3} \frac{\rho^{(j)}(k_0)}{j!}\frac{(k-k_0)^j}{(k-i)^{4N+7}},$$
satisfies the following estimates uniformly for $\zeta \in [-A, 0]$:
\begin{align}\label{fcoincideIVg}
 \frac{\partial^n f}{\partial k^n} (\zeta, k) =
\begin{cases}
O((k-k_0)^{4N+4-n}), \quad& k \to k_0, 
	\\
O(k^{-4}), &  k \to \infty, 
 \end{cases}
 \  n = 0,1,\dots, N+1.
\end{align} 

The decomposition of $r(k)$ can now be derived as follows.
For each  $\zeta \in [-A, 0]$, the map $k \mapsto \phi \equiv \phi(\zeta, k)$ where
\begin{align*}
  \phi = -i\Phi(\zeta, k) = 2\zeta k + 8k^3
\end{align*}  
is an increasing bijection  $[k_0, \infty) \to [\phi(\zeta, k_0), \infty)$.
%$$\phi(\zeta, k_0) = -\frac{(k_0^2 - 1)^3}{4k_0^2(k_0^2 +1)},$$
Hence we may define a function $F(\zeta, \phi)$ by
\begin{align}\label{FdefIVg}
F(\zeta, \phi) = \begin{cases} \frac{(k-i)^{N+3}}{(k-k_0)^{N+1}} f(\zeta, k), &  \phi \geq \phi(\zeta, k_0), \\
0, & \phi < \phi(\zeta, k_0),
\end{cases} \qquad \zeta \in [-A, 0], \ \phi \in \R.
\end{align}
For each $\zeta \in [-A, 0]$, the function $F(\zeta, \phi)$ is smooth for $\phi \neq \phi(\zeta,k_0)$ and 
\begin{align}\label{dnFdphinIVg}
\frac{\partial^n F}{\partial \phi^n}(\zeta, \phi) = \bigg(\frac{1}{\partial \phi/\partial k} \frac{\partial }{\partial k}\bigg)^n 
\bigg[\frac{(k-i)^{N+3}}{(k-k_0)^{N+1}} f(\zeta, k)\bigg], \qquad \phi \geq \phi(\zeta,k_0),
\end{align}
where
$$\frac{\partial \phi}{\partial k} = 24(k^2 -k_0^2).$$
By (\ref{fcoincideIVg}) and (\ref{dnFdphinIVg}), we have $F(\zeta, \cdot) \in C^N(\R)$ for each $\zeta$ and 
\begin{align*}
\bigg| \frac{\partial^n F}{\partial \phi^n}(\zeta, \phi)\bigg| \leq 
 \frac{C}{1 + |\phi|^{2/3}}, \qquad \phi \in (\phi(\zeta, k_0), \infty), \ \zeta \in [-A, 0], \ n = 0,1, \dots, N+1.
\end{align*}

Hence
\begin{align*}
\sup_{\zeta \in [-A, 0]} \big\|\partial_\phi^n F(\zeta, \cdot)\big\|_{L^2(\R)} < \infty, \qquad n = 0,1, \dots, N+1.
\end{align*}
In particular, $F(\zeta, \cdot)$ belongs to the Sobolev space $H^{N+1}(\R)$ for each $\zeta \in [-A,0]$.
We conclude that the Fourier transform $\hat{F}(\zeta, s)$ defined by (\ref{FhatdefIV}) satisfies (\ref{FFhatIV}) and
\begin{align}\label{x2FhatIVg}
\sup_{\zeta \in [-A,0]} \|s^{N+1} \hat{F}(\zeta, s)\|_{L^2(\R)} < \infty.
\end{align}
Equations (\ref{FFhatIV}) and (\ref{FdefIVg}) imply
$$ \frac{(k-k_0)^{N+1}}{(k-i)^{N+3}}\int_{\R} \hat{F}(\zeta, s) e^{s\Phi(\zeta,k)} ds 
= \begin{cases} f(\zeta, k), \quad &  k \geq k_0, \\
0, & k < k_0, 
 \end{cases} \ \zeta \in [-A,0].$$
We write
$$f(\zeta, k) = f_a(x, t, k) + f_r(x, t, k), \qquad \zeta \in [-A,0], \ t \geq 1, \ k \geq k_0,$$
where the functions $f_a$ and $f_r$ are defined by
\begin{align*}
& f_a(x,t,k) = \frac{(k-k_0)^{N+1}}{(k-i)^{N+3}}\int_{-\frac{t}{4}}^{\infty} \hat{F}(\zeta,s) e^{s\Phi(\zeta,k)} ds, \qquad k \in \bar{V}',  
	\\
& f_r(x,t,k) = \frac{(k-k_0)^{N+1}}{(k-i)^{N+3}}\int_{-\infty}^{-\frac{t}{4}} \hat{F}(\zeta,s) e^{s\Phi(\zeta,k)} ds,\qquad k \geq k_0.
\end{align*}
Since $\re \Phi(\zeta, k) \leq 0$ for $k \in \bar{V}'$, $f_a(x, t, \cdot)$ is continuous in $\bar{V}'$ and analytic in $V'$. 
Furthermore,
\begin{align*}\nonumber
 |f_a(x, t, k)| 
&\leq \frac{|k-k_0|^{N+1}}{|k-i|^{N+3}} \|\hat{F}(\zeta,\cdot)\|_{L^1(\R)}  \sup_{s \geq -\frac{t}{4}} e^{s \re \Phi(\zeta,k)}
\leq \frac{C|k-k_0|^{N+1}}{|k-i|^{N+3}}  e^{\frac{t}{4} |\re \Phi(\zeta,k)|} 
	\\ 
& \hspace{5cm} \zeta \in [-A,0], \ t \geq 1, \ k \in \bar{V}',
\end{align*}
and
\begin{align*}\nonumber
|f_r(x, t, k)| & \leq \frac{|k-k_0|^{N+1}}{|k-i|^{N+3}} \int_{-\infty}^{-\frac{t}{4}} s^{N+1} |\hat{F}(\zeta,s)| s^{-N-1} ds
	\\\nonumber
& \leq \frac{C}{1 + |k|^2}  \| s^{N+1} \hat{F}(\zeta,s)\|_{L^2(\R)} \sqrt{\int_{-\infty}^{-\frac{t}{4}} |s|^{-2N-2} ds}  
 	\\ 
&  \leq \frac{C}{1 + |k|^2} t^{-N-1/2}, \qquad \zeta \in [-A,0], \ t \geq 1, \ k \geq k_0.
\end{align*}
Letting
\begin{align*}
& r_{a}(x, t, k) = r_0(\zeta, k) + f_a(x, t, k), \qquad k \in \bar{V}',
	\\
& r_{r}(x, t, k) = f_r(x, t, k), \qquad k \geq k_0.
\end{align*}
we find a decomposition of $r$ for $k \in (k_0,\infty)$ with the asserted properties. We use the symmetry (\ref{rsymm}) to extend this decomposition to $k \in (-\infty, -k_0)$.
\end{proof}

Using a decomposition of $r$ as provided by Lemma \ref{decompositionlemmaleq}, we define $m^{(1)}$ by (\ref{m1def}) with $G(x,t,k)$ given by (\ref{GdefII}). By Lemma \ref{decompositionlemmaleq},
$$G(x,t,\cdot) \in I + (\dot{E}^2 \cap E^\infty)(V \cup V^*),$$
hence $m$ satisfies the RH problem (\ref{RHm}) if and only if $m^{(1)}$ satisfies the RH problem (\ref{RHm1}), where the jump matrix $v^{(1)}$ is given by (\ref{v1def}) and
$$v_4^{(1)} = \begin{pmatrix}  
 1 - |r|^2 & - r^* e^{-t\Phi}  \\
 r e^{t\Phi}  & 1
\end{pmatrix}$$
with subscripts referring to Figure \ref{Gamma1leq.pdf}.

\subsection{Local model}
As in Section \ref{sectorIVgeqsec}, we introduce the new variables $y$ and $z$ by (\ref{yzdefIV}). We now have $-C \leq y \leq 0$.
Let $\mathcal{Z}^\epsilon = (\Gamma^{(1)} \cap D_\epsilon(0))\setminus ((-\infty,-k_0) \cup (k_0, \infty))$. The map $k \mapsto z$ takes $\mathcal{Z}^\epsilon$ onto $Z \cap \{|z| < (3t)^{1/3}\epsilon\}$, where $Z$ is the contour defined in (\ref{ZdefIVg}) with $z_0 := (3t)^{1/3}k_0 = \sqrt{|y|}/2$.
We write $\mathcal{Z}^\epsilon = \cup_{j=1}^5\mathcal{Z}_j^\epsilon$, where $\mathcal{Z}_j^\epsilon$ denotes the part of $\mathcal{Z}^\epsilon$ that maps into $Z_j$, see Figure \ref{calZleq.pdf}.

\begin{figure}
\begin{center}
 \begin{overpic}[width=.4\textwidth]{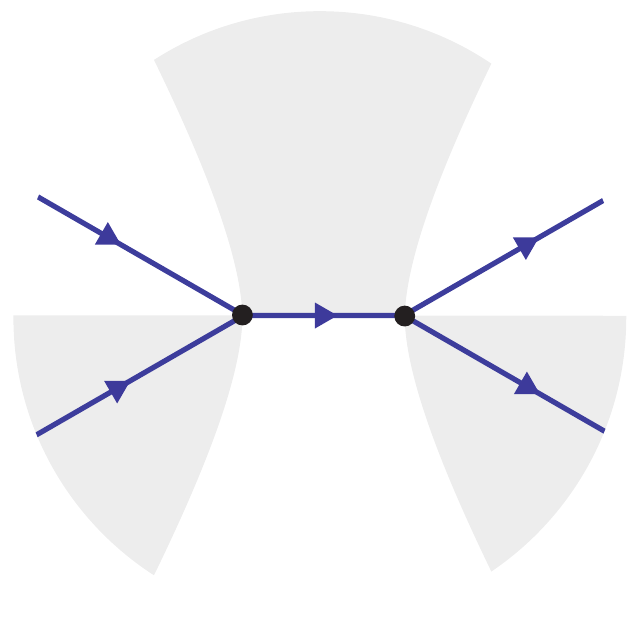}
 \put(77,65.5){\small $\mathcal{Z}_1^\epsilon$}
 \put(18,65){\small $\mathcal{Z}_2^\epsilon$}
 \put(18,32){\small $\mathcal{Z}_3^\epsilon$}
 \put(77,32){\small $\mathcal{Z}_4^\epsilon$}
 \put(47.5,54){\small $\mathcal{Z}_5^\epsilon$}
 \put(61,44){\small $k_0$}
 \put(33,44){\small $-k_0$}
 \end{overpic}
   \begin{figuretext}\label{calZleq.pdf}
      The contour $\mathcal{Z}^\epsilon = \cup_{j=1}^5\mathcal{Z}_j^\epsilon$.
      \end{figuretext}
   \end{center}
\end{figure}

The long-time asymptotics in Sector $\IV_\leq$ is related to the solution $m^Z(y,t,z)$ of the RH problem (\ref{RHmZIVg}) with $p$ given by (\ref{pNdefIV}).
Let $\IV_\leq^T := \IV_\leq \cap \{t \geq T\}$. The triple $(y,t,z_0)$ belongs to the parameter set $\mathcal{P}_T$ in (\ref{parametersetIVg}) whenever $(x,t) \in \IV_\leq^T$.
Thus, by Lemma \ref{ZlemmaIVg}, we can choose $T \geq 1$ such that
\begin{align}\label{m0defIVg}
m_0(x, t, k) := m^Z(y, t, z_0, z), \qquad k \in D_\epsilon(0),
\end{align}
is well-defined whenever $(x,t) \in \IV_\leq^T$. 
By (\ref{mZsymmIV}), $m_0$ obeys the same symmetries (\ref{msymm}) as $m$.

\begin{lemma}\label{m0lemmaIVg}
For each $(x,t) \in \IV_\leq^T$, the function $m_0(x,t,k)$ defined in (\ref{m0defIVg}) is an analytic function of $k \in D_\epsilon(0) \setminus \mathcal{Z}^\epsilon$ such that
\begin{align}\label{m0boundIVg}
|m_0(x,t,k)| \leq C, \qquad  (x,t) \in \IV_\leq^T, \ k \in D_\epsilon(0) \setminus \mathcal{Z}^\epsilon.
\end{align}
Across $\mathcal{Z}^\epsilon$, $m_0$ obeys the jump condition $m_{0+} =  m_{0-} v_0$, where the jump matrix $v_0$ satisfies, for each $1 \leq p \leq \infty$,
\begin{align}\label{v2v0estimateIVg}
 \|v^{(1)} - v_0\|_{L^p(\mathcal{Z}^\epsilon)} \leq Ct^{-\frac{N+1}{3}}, \qquad (x,t) \in \IV_\leq^T.
\end{align}	
Furthermore, 
\begin{align*}
 m_0(x,t,k)^{-1}  = I +  \sum_{j=1}^N \sum_{l=0}^N \frac{m_{0,jl}(y)}{k^j t^{(j+l)/3}} +  O(t^{-(N+1)/3})
\end{align*}  
uniformly for $(x,t) \in \IV_\leq^T$ and $k \in \partial D_\epsilon(0)$, where the coefficients $m_{0,jl}(y)$ are smooth extensions to all $y \in \R$ of the coefficients in (\ref{m0invexpansionIV}).
The function $m_0(x,t,k)$ satisfies (\ref{m0LinftyestimateIV})-(\ref{m0circleIV}) uniformly for $(x,t) \in \IV_\leq^T$, where the coefficients $\{g_j(y)\}_1^N$ are smooth functions of $y\in \R$ such that (\ref{g1explicitIV})-(\ref{g1g2g3explicit}) hold.
\end{lemma}
\begin{proof}
The analyticity of $m_0$ follows directly from the definition. The bound (\ref{m0boundIVg}) is a consequence of (\ref{mZboundedIVg}).
Moreover,
\begin{align}\label{v1minusv0IVg}
& v^{(1)} - v_0 = \begin{cases}
   \begin{pmatrix}  
 0 & 0  \\
 (r_a(x,t,k) - p(t,z)) e^{t\Phi}   & 0 
  \end{pmatrix}, & k \in \mathcal{Z}_1^\epsilon \cup \mathcal{Z}_2^\epsilon,
  	\\ 
 \begin{pmatrix}  
 0 & -(r_a^*(x,t,k) - p^*(t, z)) e^{-t\Phi} \\
0 & 0 
  \end{pmatrix}, & k \in \mathcal{Z}_3^\epsilon \cup \mathcal{Z}_4^\epsilon,
  	\\ 
  \begin{pmatrix} - |r|^2 + |p|^2
  &  -(r^*(k) - p^*(t, z)) e^{-t\Phi} \\
(r(k) - p(t,z)) e^{t\Phi} 	& 0
  \end{pmatrix}, & k \in \mathcal{Z}_5^\epsilon.  
\end{cases}
\end{align}
Since
$$\re \Phi(\zeta, k_0 + re^{\frac{\pi i}{6}}) = -4 r^2(3^{3/2} k_0 + 2r) \leq -8r^3$$
for $\zeta \leq 0$ and $r \geq 0$, we obtain 
\begin{align}\label{PhionZ1IVg}
\re \Phi(\zeta, k) \leq -8|k-k_0|^3, \qquad k \in \mathcal{Z}_1^\epsilon, \ (x,t) \in \IV_\leq^T.
\end{align}
Suppose $k \in \mathcal{Z}_1^\epsilon$. If $|k-k_0| \geq k_0$, then $|k-k_0| \geq |k|/3$, and so
$$e^{-6t|k-k_0|^3} \leq e^{-2t|k|^3}, \qquad |k-k_0| \geq k_0, \ k \in \mathcal{Z}_1^\epsilon.$$
If $|k-k_0| \leq k_0$, then $|k| \leq Ct^{-1/3}$, and so
$$e^{-6t|k-k_0|^3} \leq 1 \leq Ce^{-2t|k|^3}, \qquad |k-k_0| \leq k_0, \ k \in \mathcal{Z}_1^\epsilon.$$
We conclude that
\begin{align}\label{etPhiIVg}
e^{-\frac{3t}{4}|\re \Phi|} \leq e^{-6t|k-k_0|^3} \leq Ce^{-2t|k|^3} \leq C e^{-\frac{2}{3}|z|^3}, \qquad k \in \mathcal{Z}_1^\epsilon, \ (x,t) \in \IV_\leq.
\end{align}

On the other hand,
\begin{align}\nonumber
\sum_{j=0}^N \frac{r^{(j)}(k_0)}{j!} (k-k_0)^j 
& = \sum_{j=0}^N\sum_{l=0}^N \frac{r^{(j+l)}(0)}{j! l!} k_0^l(k-k_0)^j  + O(k_0^{N+1})
	\\\nonumber
& = \sum_{l=0}^N \sum_{s=l}^{N+l} \frac{r^{(s)}(0)}{(s-l)! l!} k_0^l(k-k_0)^{s-l}   + O(k_0^{N+1})
	\\\nonumber
&= \sum_{s=0}^{2N} r^{(s)}(0) \sum_{l=0}^s \frac{k_0^l(k-k_0)^{s-l}}{(s-l)! l!}  + O(|k-k_0|^{N+1})+ O(k_0^{N+1})
	\\\nonumber
&= \sum_{s=0}^{2N} \frac{r^{(s)}(0)}{s!}k^s  + O(|k-k_0|^{N+1})+ O(k_0^{N+1})
	\\\label{sumsumdifference}
&= \sum_{s=0}^{N} \frac{r^{(s)}(0)}{s!}k^s  + O(|k|^{N+1})
\end{align}
uniformly for $0 \leq k_0 \leq C$ and $k \in \mathcal{Z}^\epsilon$.
Equations (\ref{pNdefIV}), (\ref{raatk0}), (\ref{v1minusv0IVg}), (\ref{etPhiIVg}), and (\ref{sumsumdifference}) imply, for $k \in \mathcal{Z}_1^\epsilon$,
\begin{align*}
 |v^{(1)} - v_0| \leq &\; C|r_a(x,t,k) - p(t,z)| e^{t \re \Phi}
 \leq \bigg|r_{a}(x, t, k) - \sum_{j=0}^N \frac{r^{(j)}(k_0)}{j!} (k-k_0)^j\bigg| e^{t \re \Phi}
 	\\
& + \bigg|\sum_{j=0}^N \frac{r^{(j)}(k_0)}{j!} (k-k_0)^j - \sum_{j=0}^N \frac{r^{(j)}(0)}{j!} k^j\bigg| e^{t \re \Phi}
	\\
\leq  &\; C |k-k_0|^{N+1} e^{-\frac{3}{4}t |\re \Phi|} +  C|k|^{N+1} e^{- t |\re \Phi|} \leq C|zt^{-1/3}|^{N+1} e^{-\frac{2}{3}|z|^3}.
\end{align*}
As in the proof of Lemma \ref{m0lemmaIV}, this implies $\|v^{(1)} - v_0\|_{(L^1 \cap L^\infty)(\mathcal{Z}_1^{\epsilon})} \leq Ct^{-(N+1)/3}$.
Similar estimates apply to $\mathcal{Z}_j^{\epsilon}$ with $j = 2,3,4$. For $k \in \mathcal{Z}_5^{\epsilon}$, we have $\re \Phi = 0$, and so, by (\ref{pNdefIV}) and (\ref{v1minusv0IVg}),
\begin{align*}
 |v^{(1)} - v_0| \leq C |r(k) - p(t,z)| \leq C |k|^{N+1} \leq Ct^{-\frac{N+1}{3}}, \qquad k \in \mathcal{Z}_5^{\epsilon}.
\end{align*}
This proves (\ref{v2v0estimateIVg}). The rest of the lemma follows in the same way as Lemma \ref{m0lemmaIV}. 
\end{proof}

\subsection{The solution $\hat{m}$} 
Let $\hat{\Gamma} := \Gamma^{(1)} \cup \partial D_\epsilon(0)$ (see Figure \ref{Gammahatleq.pdf}) and define $\hat{m}(x,t,k)$ by (\ref{mhatdef}). Then $\hat{m}$ satisfies (\ref{RHmhat}) with jump matrix $\hat{v}$ given by (\ref{vhatdef}).
Write $\hat{\Gamma}$ as the union of four subcontours as follows:
$$\hat{\Gamma} = \partial D_\epsilon(0) \cup \mathcal{Z}^\epsilon \cup (\R \setminus [-k_0,k_0])\cup \hat{\Gamma}',$$
where $\hat{\Gamma}' := \Gamma^{(1)} \setminus (\R \cup \overline{D_\epsilon(0)})$.

\begin{figure}
\begin{center}
\begin{overpic}[width=.5\textwidth]{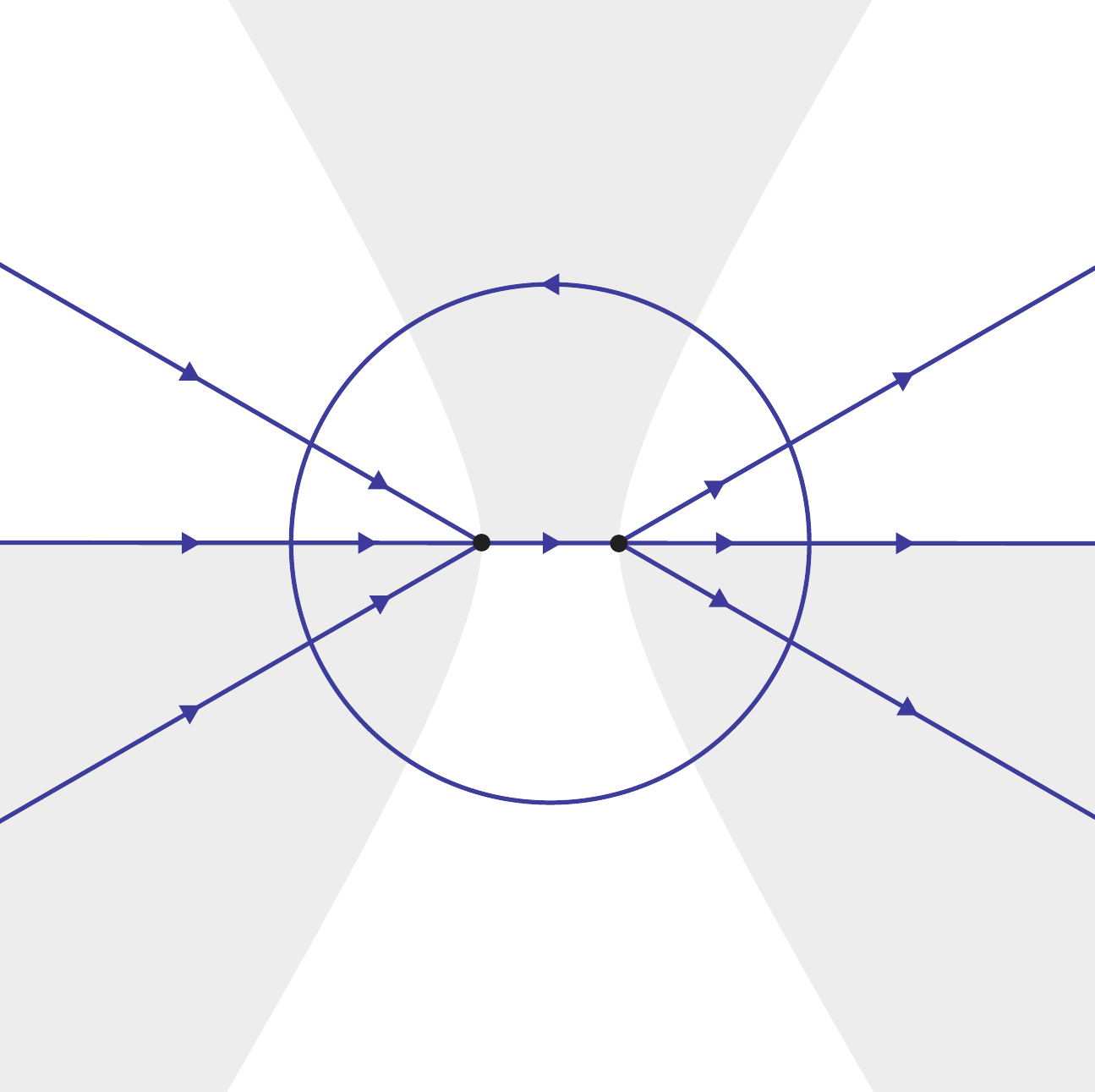}
      \put(103,48){\small $\hat{\Gamma}$}
            \put(75,46.5){\small $\epsilon$}
      \put(19.5,46.5){\small $-\epsilon$}
    \end{overpic}
     \begin{figuretext}\label{Gammahatleq.pdf}
        The contour  $\hat{\Gamma}$ in the case of Sector $\IV_\leq$.
     \end{figuretext}
     \end{center}
\end{figure}

\begin{lemma}\label{wlemmaIVg}
Let $\hat{w} = \hat{v} - I$. For each $1 \leq p \leq \infty$, the following estimates hold uniformly for $(x,t) \in \IV_\leq^T$:
\begin{subequations}\label{westimateIVg}
\begin{align}\label{westimateIVga}
& \|\hat{w}\|_{L^p(\partial D_\epsilon(0))} \leq Ct^{-1/3},
	\\\label{westimateIVgb}
& \|\hat{w}\|_{L^p(\mathcal{Z}^\epsilon)} \leq Ct^{-(N+1)/3}, 
	\\ \label{westimateIVgc}
& \|\hat{w}\|_{L^p(\R \setminus [-k_0,k_0])} \leq C t^{-N},  
	\\\label{westimateIVgd}
& \|\hat{w}\|_{L^p(\hat{\Gamma}')} \leq Ce^{-ct}.
\end{align}
\end{subequations}
\end{lemma}
\begin{proof}
The proof is analogous to the proof of Lemma \ref{wlemmaIV}.
\end{proof}

The remainder of the derivation in Sector $\IV_\leq$ now proceeds as in Sector $\IV_\geq$.

\appendix
\section{Model problem for Sector $IV_\geq$}\label{IVapp}
We first need to review the RH approach to the Painlev\'e II equation.

\subsection{Painlev\'e II Riemann--Hilbert problem}
Let $P = \cup_{n=1}^6 P_n$ denote the contour consisting of the six rays 
$$P_n = \biggl\{z \in \C \, \bigg|\, \arg z = \frac{\pi}{6} + \frac{\pi (n-1)}{3}\biggr\}, \qquad n = 1, \dots, 6,$$
oriented away from the origin, see Figure \ref{Sixrays.pdf}. %The next lemma summarizes the results we need.

\begin{figure}
\begin{center}
 \begin{overpic}[width=.35\textwidth]{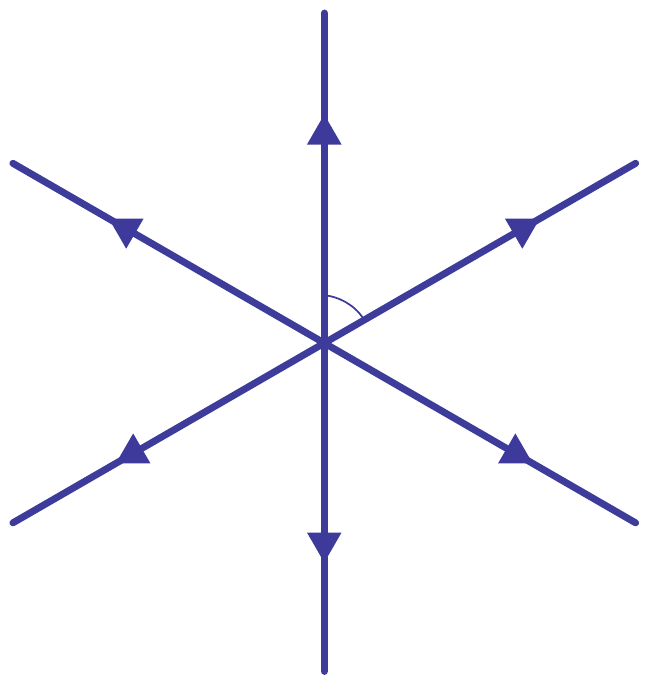}
 \put(49,58){\small $\pi/3$}
 \put(72,70.5){\small $P_1$}
 \put(37,79){\small $P_2$}
 \put(15,70.5){\small $P_3$}
 \put(15,38){\small $P_4$}
 \put(38,18){\small $P_5$}
 \put(72,38){\small $P_6$}
   \end{overpic}
     \begin{figuretext}\label{Sixrays.pdf}
       The six rays $P_n$, $n = 1, \dots, 6$, that make up the contour $P$.
     \end{figuretext}
     \end{center}
\end{figure}

\begin{proposition}[Painlev\'e II Riemann--Hilbert problem]\label{PainlevepropIV}
Let $\mathcal{S} = \{s_1, s_2, s_3\}$ be a set of complex constants such that 
\begin{align}\label{s123condition}
s_1 -s_2 +s_3 + s_1s_2s_3 = 0
\end{align}
and define the matrices $\{S_n\}_1^6$ by
$$S_n = \begin{pmatrix} 1 & 0 \\ s_n & 1 \end{pmatrix}, \quad n \text{ odd};
\qquad
S_n = \begin{pmatrix} 1 & s_n \\ 0 & 1 \end{pmatrix}, \quad n \text{ even},$$
where $s_{n+3} = - s_n$, $n = 1,2,3$. 
Then there exists a countable set $Y_{\mathcal{S}} = \{y_j\}_{j=1}^\infty \subset \C$ with $y_j \to \infty$ as $j \to \infty$, such that the classical RH problem 
\begin{align}\label{RHmP}
\begin{cases} m_+^P(y, z) =  m_-^P(y, z) e^{-i(yz + \frac{4}{3}z^3)\sigma_3} S_n e^{i(yz + \frac{4}{3}z^3)\sigma_3}, & z \in P_n\setminus \{0\}, \ n = 1, \dots, 6, \\
m^P(y, z) = I + O(z^{-1}), & z \to \infty, \\
m^P(y,z) = O(1), & z \to 0,
\end{cases} 
\end{align}
has a unique solution $m^P(y, z)$ for each $y \in \C \setminus Y_{\mathcal{S}}$. 
For each $n$, the restriction of $m^P$ to $\arg z \in (\frac{\pi(2n-3)}{6}, \frac{\pi (2n-1)}{6})$ admits an analytic continuation to $(\C \setminus Y_{\mathcal{S}}) \times \C$.
%these functions have poles at $y_j \in Y_{\mathcal{S}}$ and the coefficients of the corresponding Laurent series are entire functions of $z$. 
Moreover, there are smooth functions $\{m_j^P(y)\}_1^\infty$ of $y \in \C \setminus Y_{\mathcal{S}}$ such that, for each integer $N \geq 0$,
\begin{align}\label{mPasymptotics}
m^P(y, z) = I + \sum_{j=1}^N \frac{m_j^P(y)}{z^j} + O(z^{-N-1}), \qquad z \to \infty,
\end{align}
uniformly for $y$ in compact subsets of $\C \setminus Y_{\mathcal{S}}$ and for $\arg z \in [0,2\pi]$. The off-diagonal elements of the leading coefficient $m_1^P$ are given by
$$(m_1^P(y))_{12} = (m_1^P(y))_{21} = \frac{1}{2} u_P(y),$$
where $u_P(y) \equiv u_P(y; s_1, s_2, s_3)$ satisfies the Painlev\'e II equation (\ref{painleve2}).
The map $(s_1,s_2,s_3) \in \mathcal{S} \mapsto u_P(\cdot; s_1, s_2, s_3)$ is a bijection
\begin{align}\label{painlevebijection}
\{(s_1,s_2,s_3) \in \C^3 \,|\, s_1 -s_2 +s_3 + s_1s_2s_3 = 0\} \to \{\text{solutions of }(\ref{painleve2})\}
\end{align}
and $Y_{\mathcal{S}}$ is the set of poles of $u_P(\cdot; s_1, s_2, s_3)$.
The solution $u_P$ obeys the reality condition 
\begin{align}\label{uprealitycondition}
u_P(y; s_1, s_2, s_3) = \overline{u_P(\bar{y}; s_1, s_2, s_3)}
\end{align}
if and only if $\{s_n\}_1^3$ satisfy $s_3 = \bar{s}_1$ and $s_2 = \bar{s}_2$.
If $\mathcal{S} = (s,0,-s)$ where $s \in i\R$ and $|s| < 1$, then $Y_{\mathcal{S}} \cap \R = \emptyset$, the leading coefficient $m_1^P$ is given by
$$m_1^P(y) = \frac{1}{2} \begin{pmatrix} -i\int_y^\infty u_P(y')^2dy' & u_P(y) \\ u_P(y) & i\int_y^\infty u_P(y')^2dy'  \end{pmatrix},$$ 
and, for each $C_1 > 0$,
\begin{align}\label{mPbounded}
\sup_{y \geq -C_1} \sup_{z \in \C\setminus P} |m^P(y,z)|  < \infty.
\end{align}
\end{proposition}
\begin{proof}
The proof uses the fact that the Painlev\'e II equation (\ref{painleve2}) is the compatibility condition of the Lax pair
\begin{align}\label{painlevelax}
\begin{cases}
\partial_y\Psi + i z \sigma_3\Psi = -u_P \sigma_2 \Psi,
	\\
\partial_z\Psi + i(y + 4 z^2)\sigma_3\Psi = (-2iu_P^2 \sigma_3 - 4zu_P\sigma_2 - 2(u_P)'\sigma_1) \Psi,
\end{cases}
\end{align}
see Theorem 3.4, Theorem 4.2 and Corollary 4.4 in \cite{FIKN2006}. 
The reality condition (\ref{uprealitycondition}) can be found on p. 158 of \cite{FIKN2006}.
Employing the relation $\Psi = m^{P}e^{-i(yz+4/3 z^{3} )\sigma_{3}}$, the expression for $m_1^P(y)$ can be obtained by substituting the expansion (\ref{mPasymptotics}) of $m^P$ into the first equation in (\ref{painlevelax}) and identifying powers of $z$.
%The asymptotic formulas (\ref{mjPydecay}) and (\ref{mPasymptotics2}) follow by noticing that the main contribution comes from two small crosses centered at the critical points $z = \pm \sqrt{-y}/2$, see \cite{DZ1995}.
%See \cite{HM1980} and p. 51 of \cite{IN1986} or Theorem 5.6 of \cite{FIKN2006} for the last statement regarding $Y_{\mathcal{S}} \cap \R = \emptyset$.
\end{proof}

\subsection{Model problem for Sector $\IV_\geq$}
%We can now consider the RH problem (\ref{RHmYIV}) relevant for Sector $\IV_\geq$.
Let $Y$ denote the contour $Y = \cup_{j=1}^4 Y_j$ oriented to the right as in Figure \ref{Y.pdf}, where 
\begin{align} \nonumber
&Y_1 = \bigl\{re^{\frac{i\pi}{6}}\, \big| \, 0 \leq r < \infty\bigr\}, && Y_2 = \bigl\{re^{\frac{5i\pi}{6}}\, \big| \, 0 \leq r < \infty\bigr\},  
	\\ \label{YdefIV}
&Y_3 = \bigl\{re^{-\frac{5i\pi}{6}}\, \big| \, 0 \leq r < \infty\bigr\}, && Y_4 = \bigl\{re^{-\frac{i\pi}{6}}\, \big| \, 0 \leq r < \infty\bigr\}.
\end{align}
The long-time asymptotics in Sector $\IV_\geq$ is related to the solution $m^Y$ of the following family of RH problems parametrized by $(y,t)$:
\begin{align}\label{RHmYIV}
\begin{cases}
m^Y(y, t, \cdot) \in I + \dot{E}^2(\C \setminus Y),\\
m_+^Y(y, t, z) =  m_-^Y(y, t, z) v^Y(y, t, z) \quad \text{for a.e.} \ z \in Y,
\end{cases}
\end{align}
where the jump matrix $v^Y(y, t, z)$ has the form
\begin{align}\label{vYdefIV}
v^Y(y, t, z) = \begin{cases}
 \begin{pmatrix} 
 1	& 0 \\
p(t, z)e^{2i(y z + \frac{4z^3}{3})}  & 1
\end{pmatrix}, &  z \in Y_1 \cup Y_2, 
	\\
\begin{pmatrix} 1 & -p^*(t, z)e^{-2i(y z + \frac{4z^3}{3})}	\\
0	& 1 
\end{pmatrix}, &   z \in Y_3 \cup Y_4, 
\end{cases}
\end{align}
with the function $p$ specified below.

\begin{figure}
\begin{center}
 \begin{overpic}[width=.4\textwidth]{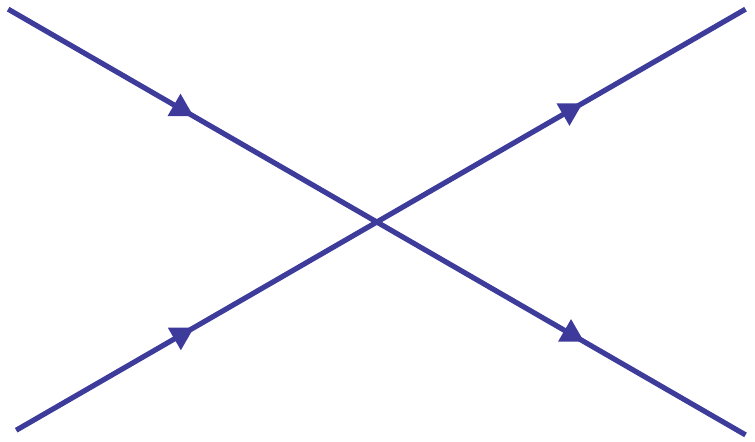}
 \put(72,48){\small $Y_1$}
 \put(22,48){\small $Y_2$}
 \put(23,6.5){\small $Y_3$}
 \put(72,7){\small $Y_4$}
 \put(48.5,23){\small $0$}
 \end{overpic}
   \bigskip
   \begin{figuretext}\label{Y.pdf}
      The contour $Y$ defined in (\ref{YdefIV}). 
      \end{figuretext}
   \end{center}
\end{figure}

\begin{lemma}[Model problem for Sector $\IV_\geq$]\label{YlemmaIV}
Let, for some integer $n \geq 0$,
\begin{align}\label{psumIV}
p(t,z) = s + \sum_{j=1}^n \frac{p_{j}z^j}{t^{j/3}},
\end{align}
be a polynomial in $z t^{-1/3}$ with coefficients $s \in \{ir \, |  -1 < r < 1\}$ and $\{p_j\}_1^n \subset \C$.

\begin{enumerate}[$(a)$]
\item There is a $T \geq 1$ such that the RH problem (\ref{RHmYIV}) with jump matrix $v^Y$ given by (\ref{vYdefIV}) has a unique solution $m^Y(y, t, z)$ whenever $y \geq 0$ and $t \geq T$. 

\item For each integer $N \geq 1$, there are smooth functions $\{m_{jl}^Y(y)\}$ of $y \in [0,\infty)$ such that
\begin{align}\label{mYasymptoticsIV}
&  m^Y(y, t, z) = I + \sum_{j=1}^N \sum_{l=0}^N \frac{m_{jl}^Y(y)}{z^j t^{l/3}}  + O\biggl(\frac{t^{-(N+1)/3}}{|z|} + \frac{1}{|z|^{N+1}}\biggr), \qquad z \to \infty,
\end{align}
uniformly with respect to $\arg z \in [0, 2\pi]$, $y \geq 0$, and $t \geq T$.

\item $m^Y$ obeys the bound
\begin{align}\label{mYboundedIV}
\sup_{y \geq 0} \sup_{t \geq T} \sup_{z \in \C\setminus Y} |m^Y(y, t, z)|  < \infty.
\end{align}

\item $m^Y$ obeys the symmetry
\begin{subequations}\label{mYsymmIV}
\begin{align}
m^Y(y, t, z) = \sigma_1\overline{m^Y(y, t, \bar{z})} \sigma_1.
\end{align}
If $p(t,z) = -\overline{p(t,-\bar{z})}$, then it also obeys the symmetry
\begin{align}
m^Y(y, t, z) = \sigma_1\sigma_3m^Y(y, t, -z) \sigma_3\sigma_1.
\end{align}
\end{subequations}

\item The leading coefficient in (\ref{mYasymptoticsIV}) is given by
\begin{align}\label{m10Yexplicit}
m_{10}^Y(y) =\frac{1}{2} \begin{pmatrix} -i\int_y^\infty u_P(y'; s, 0, -s)^2dy' & u_P(y; s, 0, -s) \\ u_P(y; s, 0, -s) & i\int_y^\infty u_P(y'; s, 0, -s)^2dy'  \end{pmatrix},
\end{align}
where $u_P(\cdot; s, 0, -s)$ denotes the smooth real-valued solution of Painlev\'e II corresponding to $(s,0,-s)$ via (\ref{painlevebijection}).

\item If $s = 0$, $p_1 \in \R$, and $p_2 \in i\R$, then the leading coefficients are given by
\begin{align}\nonumber
& m_{10}^Y(y) = 0,
	\\\nonumber
& m_{11}^Y(y) = \frac{p_1}{4}\Ai'(y)\sigma_1,
	\\\nonumber
& m_{12}^Y(y) = \frac{p_1^2}{8i} \bigg(\int_y^{\infty} (\Ai'(y'))^2dy'\bigg) \sigma_3 
+ \frac{p_2}{8i} \Ai''(y) \sigma_1,
	\\ \label{mYexplicitIV}
& m_{21}^Y(y) = -\frac{p_1}{8i}\Ai''(y) \sigma_3\sigma_1.
\end{align}
\end{enumerate}
\end{lemma}
\begin{proof}
Let $u_P(y; s, 0, -s)$ denote the solution of the Painlev\'e II equation (\ref{painleve2}) corresponding to $(s,0,-s)$ according to (\ref{painlevebijection}). Since $s \in i\R$ and $|s| < 1$, $u_P(y; s, 0, -s)$ is a smooth real-valued function of $y \in \R$ (see Proposition \ref{PainlevepropIV}). Let $m^P(y,z) \equiv m^P(y,z;s,0,-s)$ be the corresponding solution of the RH problem (\ref{RHmP}). Then $m^P(y,z)$ solves the RH problem obtained from (\ref{RHmYIV}) by replacing the polynomial $p(t,z)$ on the right-hand side of (\ref{vYdefIV}) with its leading term $s$. 

The function $m^Y$ satisfies (\ref{RHmYIV}) iff $\hat{m}^Y := m^Y (m^P)^{-1}$ satisfies
\begin{align}\label{RHmYhatIV}
\begin{cases}
\hat{m}^Y(y, t, \cdot) \in I + \dot{E}^2(\C \setminus Y),\\
\hat{m}_+^Y(y, t, z) =  \hat{m}_-^Y(y, t, z) \hat{v}^Y(y, t, z) \quad \text{for a.e.} \ z \in Y,
\end{cases}
\end{align}
where the jump matrix obeys the relation $\hat{v}^Y -I = m_{-}^P (v^Y - v^P)(m_+^P)^{-1}$. 
Letting $\hat{w}^Y := \hat{v}^Y - I$, we can write
\begin{align}\label{wYexpansionIV}
\hat{w}^Y(y, t, z) = \frac{\hat{w}_1^Y(y,z)}{t^{1/3}} + \cdots + \frac{\hat{w}_n^Y(y,z)}{t^{n/3}},
\end{align}
where
\begin{align}\label{wjYIV}
\hat{w}_j^Y(y, z) = \begin{cases}
m_-^P \begin{pmatrix} 
 0	& 0 \\
p_{j}z^je^{2i(y z + \frac{4z^3}{3})}  & 0 
\end{pmatrix}(m_+^P)^{-1}, &  z \in Y_1 \cup Y_2, 
	\\
m_-^P\begin{pmatrix} 0 & - \bar{p}_{j}z^je^{-2i(y z + \frac{4z^3}{3})} 	\\
0	& 0 
\end{pmatrix}(m_+^P)^{-1}, &   z \in Y_3 \cup Y_4.
\end{cases}
\end{align}

We next note that
\begin{align}\label{eonYIV}
|e^{\pm 2i(y z + \frac{4z^3}{3})}|
\leq e^{-\frac{8}{3}|z|^3}, \qquad y \geq 0, \  z \in Y,
\end{align}
where the plus (minus) sign applies for $z \in Y_1 \cup Y_2$ ($z \in Y_3 \cup Y_4$). 
The estimates (\ref{mPbounded}) and (\ref{eonYIV}) give, for any integer $m \geq 0$,
\begin{align*}
|z^m\hat{w}_j^Y(y,z)| \leq &\; C |z|^{m+j} e^{-\frac{8}{3}|z|^3} \leq Ce^{-c|z|^3}, \qquad
y \geq 0, \ z \in Y, \ j = 1, \dots, n,
\end{align*}
and hence, for any integer $m \geq 0$ and any $1 \leq p \leq \infty$,
\begin{align}\label{wYLpestIV}
\begin{cases}
\|z^m\hat{w}_j^Y(y,z)\|_{L^p(Y)} \leq C, & j = 1, \dots, n,
	\\
\|z^m\hat{w}^Y(y,t,z)\|_{L^p(Y)} \leq Ct^{-1/3}, \quad & t \geq 1,
\end{cases} \;\; y \geq 0.
\end{align}
In particular,
$$\|\mathcal{C}_{\hat{w}^Y(y,t,\cdot)}^Y\|_{\mathcal{B}(L^2(Y))} \leq C \|\hat{w}^Y\|_{L^\infty(Y)} \leq Ct^{-1/3}, \qquad y \geq 0, \ t \geq 1.$$
Hence there exists a $T \geq 1$ such that the RH problem (\ref{RHmYhatIV}) has a unique solution $\hat{m}^Y \in I + \dot{E}^2(\C \setminus \hat{\Gamma})$ whenever $y \geq 0$ and $t \geq T$. This solution is given by 
\begin{align}\label{mYrepresentationIV}
\hat{m}^Y(y, t, z) = I + \mathcal{C}^Y(\hat{\mu}^Y \hat{w}^Y) = I + \frac{1}{2\pi i}\int_{Y} (\hat{\mu}^Y \hat{w}^Y)(y, t, s) \frac{ds}{s - z},
\end{align}
where $\hat{\mu}^Y(x, t, \cdot) \in I + L^2(Y)$ is defined by 
\begin{align}\label{muYdefIV}
\hat{\mu}^Y = I + (I - \mathcal{C}^Y_{\hat{w}^Y})^{-1}\mathcal{C}^Y_{\hat{w}^Y}I.
\end{align} 

Let $N \geq 1$ be an integer.
Using that
$$\mathcal{C}^Y_{\hat{w}^Y} = \frac{\mathcal{C}^Y_{\hat{w}_1^Y}}{t^{1/3}} + \frac{\mathcal{C}^Y_{\hat{w}_2^Y}}{t^{2/3}} + \cdots + \frac{\mathcal{C}^Y_{\hat{w}_n^Y}}{t^{n/3}},$$
it follows from (\ref{wYLpestIV}) and (\ref{muYdefIV}) that 
\begin{align}\nonumber
\hat{\mu}^Y(y, t,z) & = \sum_{r=0}^{N} (\mathcal{C}^Y_{\hat{w}^Y})^rI 
+ (I - \mathcal{C}^Y_{\hat{w}^Y})^{-1}(\mathcal{C}^Y_{\hat{w}^Y})^{N+1}I	
	\\\label{muYexpansionIV}
& = I + \sum_{j=1}^N \frac{\hat{\mu}_j^Y(y,z)}{t^{j/3}} + \frac{\hat{\mu}_{err}^Y(y,t,z)}{t^{(N+1)/3}},
\end{align}
where the coefficients $\hat{\mu}_j^Y$ and $\hat{\mu}_{err}^Y$ satisfy
\begin{align}\label{mujYestIV}
\begin{cases}
 \|\hat{\mu}_j^Y(y,\cdot)\|_{L^2(Y)} \leq C, & j = 1, \dots, N,
	\\
 \|\hat{\mu}_{err}^Y(y,t,\cdot)\|_{L^2(Y)} \leq C, \quad & t \geq T,
 \end{cases} \;\; y \geq 0.
\end{align}
Indeed,  to obtain (\ref{mujYestIV}) we note that $\hat{\mu}_{j}^Y$ is a sum of terms of the form $\mathcal{C}^Y_{\hat{w}_{j_1}^Y}\cdots \mathcal{C}^Y_{\hat{w}_{j_r}^Y}I$
where $j_1 + \cdots j_r = j$, which, by (\ref{wYLpestIV}), can be estimated as
\begin{align}\nonumber
\|\mathcal{C}^Y_{\hat{w}_{j_1}^Y}\cdots \mathcal{C}^Y_{\hat{w}_{j_{r}}^Y}I\|_{L^2(Y)} & \leq 
C\|\mathcal{C}_{\hat{w}_{j_1}^Y}^Y\|_{\mathcal{B}(L^2(Y))}\cdots
\|\mathcal{C}_{\hat{w}_{j_{r-1}}^Y}^Y\|_{\mathcal{B}(L^2(Y))}
\|\hat{w}_{j_r}^Y\|_{L^2(Y)}
	\\\label{CYhatwj1jrIV}
& \leq 
C\|\hat{w}_{j_1}^Y\|_{L^\infty(Y)}\cdots
\|\hat{w}_{j_{r-1}}^Y\|_{L^\infty(Y)}
\|\hat{w}_{j_r}^Y\|_{L^2(Y)} \leq C;
\end{align}
the coefficient $\hat{\mu}_{err}^Y$ involves terms of the same form (but with $j_1 + \cdots j_r \geq N+1$) which can be estimated in the same way, as well as terms of the form
$$(I - \mathcal{C}^Y_{\hat{w}^Y})^{-1}\bigg(\prod_{s=1}^{N+1}\mathcal{C}^Y_{\hat{w}_{j_s}^Y}\bigg)I$$	
which, employing (\ref{CYhatwj1jrIV}), can be estimated as
$$\Big\|(I - \mathcal{C}^Y_{\hat{w}^Y})^{-1}\bigg(\prod_{s=1}^{N+1}\mathcal{C}^Y_{\hat{w}_{j_s}^Y}\bigg)I\Big\|_{L^2(Y)} \leq
C\|(I - \mathcal{C}^Y_{\hat{w}^Y})^{-1}\|_{\mathcal{B}(L^2(Y))} \leq C.$$	

Substituting the expansions (\ref{wYexpansionIV}) and (\ref{muYexpansionIV}) into the representation (\ref{mYrepresentationIV}) for $\hat{m}^Y$, we infer that the following formula holds uniformly for $y \geq 0$ and $t \geq T$ as $z \in \C \setminus Y$ goes to infinity in any nontangential sector:
\begin{align*}\nonumber
\hat{m}^Y&(y, t, z) =  I - \sum_{j=1}^{N} \frac{1}{2\pi i z^j}\int_{Y} s^{j-1}\hat{\mu}^Y \hat{w}^Y ds
+ \frac{1}{2\pi i}\int_{Y} \frac{s^N\hat{\mu}^Y \hat{w}^Y}{z^N(s-z)}ds
	\\\nonumber
= &\; I - \sum_{j=1}^{N} \frac{1}{2\pi i z^j}\int_{Y} s^{j-1}\bigg(I + \sum_{a=1}^N \frac{\hat{\mu}_a^Y(y,z)}{t^{a/3}} + \frac{\hat{\mu}_{err}^Y(y,t,z)}{t^{(N+1)/3}}\bigg)
\bigg(\sum_{b=1}^n \frac{\hat{w}_b^Y(y,z)}{t^{b/3}}\bigg)ds
	\\
& + O(|z|^{-N-1} \|\hat{\mu}^Y\|_{L^2(Y)} \|s^N\hat{w}^Y\|_{L^2(Y)}),
\end{align*}
i.e., utilizing (\ref{wYLpestIV}) and (\ref{mujYestIV}) and setting $\hat{w}_i^Y \equiv 0$ for $i \geq n+1$ if $N \geq n + 1$,
\begin{align}\nonumber
\hat{m}^Y(y, t, z) = &\; I - \sum_{j=1}^{N} \frac{1}{2\pi i z^j} \bigg\{\sum_{l=1}^N t^{-l/3} \int_{Y} s^{j-1} \bigg(\hat{w}_l^Y + \sum_{i=1}^{l-1} \hat{\mu}_{l-i}^Y \hat{w}_i^Y\bigg)ds 
 + O\big(t^{-\frac{N+1}{3}}\big)\bigg\}
	\\\label{mYexpansionIV}
& + O\big(|z|^{-N-1}t^{-1/3}\big).
\end{align}
Repeating the above steps with $Y$ replaced by a slightly deformed contour $\tilde{Y}$, we see that in fact the condition that $z$ lies in a nontangential sector can be dropped in (\ref{mYexpansionIV}). We deduce that 
\begin{align}\label{mhatYasymptoticsIV}
&  \hat{m}^Y(y, t, z) = I + \sum_{j=1}^N \sum_{l=1}^N \frac{\hat{m}_{jl}^Y(y)}{z^j t^{l/3}}  + O\biggl(\frac{t^{-(N+1)/3}}{|z|} + \frac{t^{-1/3}}{|z|^{N+1}}\biggr), \qquad z \to \infty,
\end{align}
uniformly with respect to $\arg z \in [0, 2\pi]$, $y \geq 0$, and $t \geq T$, where
$$\hat{m}_{jl}^Y(y) = - \frac{1}{2\pi i} \int_{Y} s^{j-1} \bigg(\hat{w}_l^Y + \sum_{i=1}^{l-1}\hat{\mu}_{l-i}^Y \hat{w}_i^Y\bigg)(y,s)ds, \qquad 1 \leq j,l \leq N.$$ 
The smoothness of $\hat{m}_{jl}^Y(y)$ follows from the fact (see (\ref{wjYIV}) and (\ref{muYexpansionIV})) that $y \mapsto (\cdot)^m\hat{w}_j^Y(y, \cdot)$ and $y \mapsto \hat{\mu}_j^Y(y, \cdot)$ are smooth maps $[0,\infty) \to L^p(Y)$, $1 \leq p \leq \infty$, and $[0,\infty) \to L^2(Y)$, respectively. 
The expansion (\ref{mYasymptoticsIV}) of $m^Y = \hat{m}^Y m^P$ follows from the expansions (\ref{mPasymptotics}) and (\ref{mhatYasymptoticsIV}).
The bound (\ref{mYboundedIV}) follows from (\ref{mPbounded}) and (\ref{mhatYasymptoticsIV}) and the fact that the contour can be deformed. 
The symmetries (\ref{mYsymmIV}) follow from the analogous symmetries for $v^Y$, i.e., 
$v^Y(y,t,z) = \sigma_1 \overline{v^Y(y,t,\bar{z})}^{-1} \sigma_1$ and, if $p(t,z) = -p^*(t,-z)$, $v^Y(y,t,z) = \sigma_1 \sigma_3 v^Y(y,t, -z)^{-1} \sigma_3 \sigma_1$.
 
Finally, to prove (\ref{mYexplicitIV}), assume that $s = 0$, $p_1 \in \R$, and $p_2 \in i\R$. 
In this case, $m^P \equiv I$ and (\ref{wjYIV}) gives the following expressions for $w_1^Y = \hat{w}_1^Y$ and $w_2^Y=\hat{w}_2^Y$:
\begin{align}
& w_1^Y = \begin{pmatrix} 
 0	& -p_1ze^{-2i(y z + \frac{4z^3}{3})} 1_{Y_3 \cup Y_4}(z) \\
p_1z e^{2i(y z + \frac{4z^3}{3})} 1_{Y_1 \cup Y_2}(z) & 0 
\end{pmatrix},
	\\ \label{w2YIV}
& w_2^Y = \begin{pmatrix} 
 0	& p_2z^2e^{-2i(y z + \frac{4z^3}{3})} 1_{Y_3 \cup Y_4}(z)  \\
p_2z^2e^{2i(y z + \frac{4z^3}{3})} 1_{Y_1 \cup Y_2}(z) & 0
\end{pmatrix},
\end{align}
where $1_{A}(z)$ denotes the characteristic function of the set $A \subset \C$.
Now
\begin{align*}
\int_{Y_1 \cup Y_2}  e^{2i(yz + \frac{4z^3}{3})} dz 
= \int_{Y_3 \cup Y_4} e^{-2i(yz + \frac{4z^3}{3})} dz
= \pi \Ai(y)
\end{align*}
so that, differentiating $j$ times with respect to $y$,
\begin{align}\label{intAiryIV}
\int_{Y_1 \cup Y_2} z^j e^{2i(yz + \frac{4z^3}{3})} dz 
= (-1)^j \int_{Y_3 \cup Y_4} z^j e^{-2i(yz + \frac{4z^3}{3})} dz
= \frac{\pi \Ai^{(j)}(y)}{(2i)^j}
\end{align}
for each integer $j \geq 0$. The coefficients $m_{11}^Y = \hat{m}_{11}^Y$ and $m_{21}^Y = \hat{m}_{21}^Y$ can now be computed:
\begin{align*}
 m_{11}^Y(y) = & -\frac{1}{2\pi i} \int_{Y} w_1^Y dz
	\\
= & - \frac{p_1}{2\pi i} \begin{pmatrix} 0 & - \int_{Y_3 \cup Y_4}  z e^{-2i(yz + \frac{4z^3}{3})} dz  \\
 \int_{Y_1 \cup Y_2} z e^{2i(yz + \frac{4z^3}{3})} dz & 0 
 \end{pmatrix}
	\\
= &  -\frac{p_1}{2\pi i} \frac{\pi \Ai'(y)}{2i} 
\begin{pmatrix} 0 & 1   \\
1 & 0  \end{pmatrix},
	\\
m_{21}^Y(y) = & -\frac{1}{2\pi i}\int_{Y} z w_1^Y  dz
	\\
= & -\frac{p_1}{2\pi i} \begin{pmatrix} 0 & -\int_{Y_3 \cup Y_4}  z^2 e^{-2i(yz + \frac{4z^3}{3})} dz   \\
\int_{Y_1 \cup Y_2} z^2 e^{2i(yz + \frac{4z^3}{3})} dz & 0 \end{pmatrix}
	\\
= &\; \frac{p_1}{2\pi i} \frac{\pi \Ai''(y)}{4}  \begin{pmatrix} 0 & -1   \\
1 & 0  \end{pmatrix}.
\end{align*}

It only remains to derive the expression for $m_{12}^Y$. Using that
\begin{align*}
\mu_1^Y(y,z) & = \mathcal{C}^Y_{w_1^Y}I
= \mathcal{C}^Y_-(w_1^Y)
= \frac{1}{2\pi i}\int_Y \frac{w_1^Y(y,s)ds}{s-z_-}
	\\
& = \frac{p_1}{2\pi i} \begin{pmatrix}0 & -\int_{Y_3 \cup Y_4} \frac{se^{-2i(ys + \frac{4s^3}{3})} ds}{s-z_-} \\
\int_{Y_1 \cup Y_2} \frac{se^{2i(ys + \frac{4s^3}{3})} ds}{s-z_-} & 0 \end{pmatrix},
\end{align*}
we find that the matrix-valued function $F(y)$ defined by $F(y) = \int_{Y} \mu_1^Y w_1^Y dz$ satisfies
\begin{align*}
F(y) 
& = -\frac{p_1^2}{2\pi i} \begin{pmatrix} F_{1}(y) & 0 \\
0 & F_{2}(y)
\end{pmatrix},
\end{align*}
where the diagonal entries are given by 
\begin{align*}
& F_1(y) = \int_{Y_1 \cup Y_2} \bigg(  \int_{Y_3 \cup Y_4}  \frac{sze^{-2i(ys + \frac{4s^3}{3})}e^{2i(yz + \frac{4z^3}{3})}}{s-z} ds \bigg)dz,
	\\
& F_2(y) = \int_{Y_3 \cup Y_4} \bigg( \int_{Y_1 \cup Y_2} \frac{sze^{2i(ys + \frac{4s^3}{3})}e^{-2i(yz + \frac{4z^3}{3})}}{s-z} ds \bigg) dz.
\end{align*}
Fubini's theorem implies that $F_2(y) = -F_1(y)$ and so
\begin{align*}
F(y) = -\frac{p_1^2}{2\pi i} 
F_{1}(y)\sigma_3.
\end{align*}
Differentiation with respect to $y$ yields
\begin{align*}
F'(y) & = \frac{p_1^2}{\pi} 
\bigg(\int_{Y_1 \cup Y_2} ze^{2i(yz + \frac{4z^3}{3})} dz\bigg)\bigg( \int_{Y_3 \cup Y_4}  s e^{-2i(ys + \frac{4s^3}{3})} ds\bigg)\sigma_3
	\\
& = \frac{p_1^2}{\pi} 
\bigg(\frac{\pi \Ai'(y)}{2i}\bigg)\bigg(-\frac{\pi \Ai'(y)}{2i}\bigg)
\sigma_3,
\end{align*}
where we have used (\ref{intAiryIV}) with $j = 1$.
Using that $F(y) \to 0$ as $y\to \infty$, we arrive at
$$\int_{Y} \mu_1^Y(y,z) w_1^Y(y,z)dz = \int_{+\infty}^y F'(y') dy' = \frac{\pi p_1^2}{4} \int_{+\infty}^y \Ai'(y')^2dy'
\sigma_3.$$
Taking advantage of this formula, recalling (\ref{w2YIV}), and using (\ref{intAiryIV}) again, the coefficient $m_{12}^Y$ is easily computed:
\begin{align*}
 m_{12}^Y(y) = & -\frac{1}{2\pi i}\int_{Y} (w_2^Y(y,z) + \mu_1^Y(y,z) w_1^Y(y,z)) dz
	\\
 = &
 - \frac{p_2}{2\pi i} \begin{pmatrix} 0 & \int_{Y_3 \cup Y_4}  z^2 e^{-2i(yz + \frac{4z^3}{3})} dz  \\
 \int_{Y_1 \cup Y_2} z^2 e^{2i(yz + \frac{4z^3}{3})} dz & 0 \end{pmatrix}
	\\
& -\frac{1}{2\pi i}\frac{\pi p_1^2}{4} \int_{+\infty}^y (\Ai'(y'))^2dy' \sigma_3
	\\	
 = &\; \frac{p_2}{2\pi i} \frac{\pi}{4} \Ai''(y) \sigma_1
 -\frac{p_1^2}{8i} \int_{+\infty}^y (\Ai'(y'))^2dy' \sigma_3.
\end{align*}
This completes the proof of the lemma. \end{proof}

\section{Model problem for Sector $IV_\leq$}\label{IVgeqapp}
For each $z_0 \geq 0$, let $Z \equiv Z(z_0)$ denote the contour $Z = \cup_{j=1}^5 Z_j$ oriented as in Figure \ref{Z.pdf}, where 
\begin{align} \nonumber
&Z_1 = \bigl\{z_0+ re^{\frac{i\pi}{6}}\, \big| \, 0 \leq r < \infty\bigr\}, && Z_2 = \bigl\{-z_0 + re^{\frac{5i\pi}{6}}\, \big| \, 0 \leq r < \infty\bigr\},  
	\\ \nonumber
&Z_3 = \bigl\{-z_0 + re^{-\frac{5i\pi}{6}}\, \big| \, 0 \leq r < \infty\bigr\}, && Z_4 = \bigl\{z_0 + re^{-\frac{i\pi}{6}}\, \big| \, 0 \leq r < \infty\bigr\}, 
	\\ \label{ZdefIVg}
& Z_5 = \bigl\{r\, \big|  -z_0 \leq r \leq z_0\bigr\}.
\end{align}
The long-time asymptotics in Sector $\IV_\leq$ is related to the solution $m^Z$ of the following family of RH problems parametrized by $y \leq 0$, $t \geq 0$, and $z_0 \geq 0$:
\begin{align}\label{RHmZIVg}
\begin{cases}
m^Z(y, t, z_0, \cdot) \in I + \dot{E}^2(\C \setminus Z),\\
m_+^Z(y, t, z_0, z) =  m_-^Z(y, t, z_0, z) v^Z(y, t, z_0, z) \quad \text{for a.e.} \ z \in Z,
\end{cases}
\end{align}
where the jump matrix $v^Z(y, t, z_0, z)$ is defined by
\begin{align}\label{vZdefIVg}
v^Z(y, t, z_0, z) = \begin{cases}
 \begin{pmatrix} 
 1	& 0 \\
p(t,z) e^{2i(y z + \frac{4z^3}{3})}  & 1 
\end{pmatrix}, &  z \in Z_1 \cup Z_2, 
	\\
\begin{pmatrix} 1 & -p^*(t, z)e^{-2i(y z + \frac{4z^3}{3})}	\\
0	& 1 
\end{pmatrix}, &   z \in Z_3 \cup Z_4, 
  	\\
\begin{pmatrix} 1 - |p(t, z)|^2 & -p^*(t, z)e^{-2i(y z + \frac{4z^3}{3})} \\
p(t,z) e^{2i(y z + \frac{4z^3}{3})}	& 1 \end{pmatrix}, &  z \in Z_5,
\end{cases}
\end{align}
with $p(t,z)$ given by (\ref{psumIV}). 
Define the parameter subset $\mathcal{P}_T$ of $\R^3$ by
\begin{align}\label{parametersetIVg}
\mathcal{P}_T = \{(y,t,z_0) \in \R^3 \, | \, -C_1 \leq y \leq 0, \, t \geq T, \, \sqrt{|y|}/2 \leq z_0 \leq C_2\},
\end{align}
where $C_1,C_2 > 0$ are constants. 

\begin{figure}
\begin{center}
 \begin{overpic}[width=.55\textwidth]{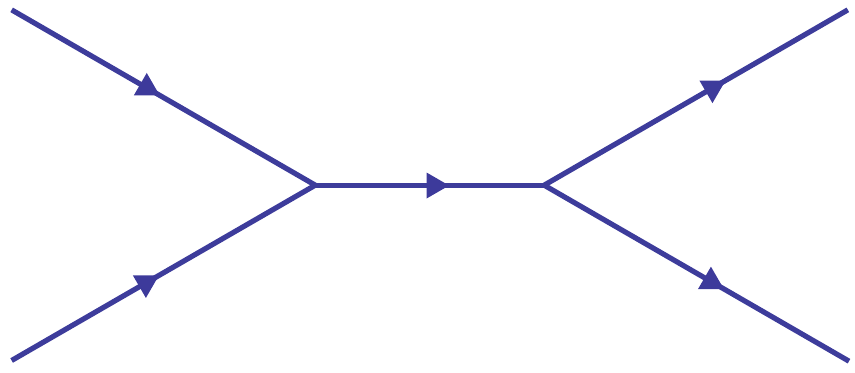}
 \put(79,36){\small $Z_1$}
 \put(17,36){\small $Z_2$}
 \put(17.5,6){\small $Z_3$}
 \put(78,6){\small $Z_4$}
 \put(49,25){\small $Z_5$}
 \put(62,18){\small $z_0$}
 \put(34,18){\small $-z_0$}
 \end{overpic}
   \begin{figuretext}\label{Z.pdf}
      The contour $Z$. 
      \end{figuretext}
   \end{center}
\end{figure}

\begin{lemma}[Model problem in Sector $\IV_\leq$]\label{ZlemmaIVg}
Let $p(t,z)$ be of the form (\ref{psumIV}) with $s \in \{ir \, |  -1 < r < 1\}$ and $\{p_{j}\}_1^n \subset \C$.

\begin{enumerate}[$(a)$]
\item There is a $T \geq 1$ such that the RH problem (\ref{RHmZIVg}) with jump matrix $v^Z$ given by (\ref{vZdefIVg}) has a unique solution $m^Z(y, t, z_0, z)$ whenever $(y,t,z_0) \in \mathcal{P}_T$. 

\item For each integer $N \geq 1$, 
\begin{align}\label{mZasymptoticsIVg}
&  m^Z(y, t, z_0, z) = I + \sum_{j=1}^N \sum_{l=0}^N \frac{m_{jl}^Y(y)}{z^j t^{l/3}}  + O\biggl(\frac{t^{-(N+1)/3}}{|z|} + \frac{1}{|z|^{N+1}}\biggr)
\end{align}
uniformly with respect to $\arg z \in [0, 2\pi]$ and $(y,t,z_0) \in \mathcal{P}_T$ as $z \to \infty$, where $\{m_{jl}^Y(y)\}$ are smooth functions of $y \in \R$ which coincide with the functions in (\ref{mYasymptoticsIV}) for $y \geq 0$. 

\item $m^Z$ obeys the bound
\begin{align}\label{mZboundedIVg}
\sup_{(y,t,z_0) \in \mathcal{P}_T} \sup_{z \in \C\setminus Z} |m^Z(y, t, z_0, z)|   < \infty.
\end{align}

\item $m^Z$ obeys the symmetry
\begin{subequations}\label{mZsymmIV}
\begin{align}
m^Z(y, t, z_0, z) = \sigma_1\overline{m^Z(y, t, z_0, \bar{z})} \sigma_1.
\end{align}
If $p(t,z) = -\overline{p(t,-\bar{z})}$, then it also obeys the symmetry
\begin{align}
m^Z(y, t, z_0, z) = \sigma_1\sigma_3m^Z(y, t, z_0, -z)\sigma_3 \sigma_1.
\end{align}
\end{subequations}

\item The leading coefficient $m_{10}^Y(y)$ in (\ref{mZasymptoticsIVg}) is given by (\ref{m10Yexplicit}); if $s = 0$, $p_1 \in \R$, and $p_2 \in i\R$, then (\ref{mYexplicitIV}) holds.

\end{enumerate}
\end{lemma}
\begin{proof}
As in the proof of Lemma \ref{YlemmaIV}, we let $u_P(y; s, 0, -s)$ denote the smooth real-valued solution of (\ref{painleve2}) corresponding to $(s,0,-s)$ and we let $m^P(y,z) \equiv m^P(y,z;s,0,-s)$ be the corresponding solution of the Painlev\'e II RH problem (\ref{RHmP}). 
Let $\{V_j\}_1^4$ denote the open subsets of $\C$ displayed in Figure \ref{Zfourrays.pdf}.
Defining $m^{P1}(y,z_0,z)$ by 
$$m^{P1}(y,z) = m^{P}(y,z) \times \begin{cases} 
\begin{pmatrix} 1 & 0 \\ s e^{2i(yz + \frac{4 z^3}{3})} & 1 \end{pmatrix}, & z \in V_1 \cup V_2, 
	\\
\begin{pmatrix} 1 & \bar{s} e^{-2i(yz + \frac{4 z^3}{3})} \\ 0 & 1 \end{pmatrix}, & z \in V_3 \cup V_4, 
\end{cases}$$
we see that $m^{P1}$ satisfies the RH problem
\begin{align*}
\begin{cases}
m^{P1}(y, z_0, \cdot) \in I + \dot{E}^2(\C \setminus Z),\\
m_+^{P1}(y, z_0, z) =  m_-^{P1}(y, z_0, z) v^{P1}(y, z_0, z) \quad \text{for a.e.} \ z \in Z,
\end{cases}
\end{align*}
where the jump matrix $v^{P1}(y, z_0, z)$ is defined by
\begin{align*}
v^{P1}(y, z_0, z) = \begin{cases}
 \begin{pmatrix} 
 1	& 0 \\
s e^{2i(y z + \frac{4z^3}{3})}  & 1 
\end{pmatrix}, &  z \in Z_1 \cup Z_2, 
	\\
\begin{pmatrix} 1 & -\bar{s} e^{-2i(y z + \frac{4z^3}{3})}	\\
0	& 1 
\end{pmatrix}, &   z \in Z_3 \cup Z_4, 
  	\\
\begin{pmatrix} 1 - |s|^2 & -\bar{s} e^{-2i(y z + \frac{4z^3}{3})} \\
s e^{2i(y z + \frac{4z^3}{3})}	& 1 \end{pmatrix}, &  z \in Z_5.
\end{cases}
\end{align*}
In other words, $m^{P1}(y,z_0, z)$ solves the RH problem obtained from (\ref{RHmZIVg}) by replacing the polynomial $p(t,z)$ on the right-hand side of (\ref{vZdefIVg}) with its leading term $s$. The bound (\ref{mPbounded}) implies that, for any choice of $C_1, C_2 > 0$,
\begin{align}\label{mP1bounded}
\sup_{-C_1 \leq y \leq 0} \sup_{0 \leq z_0 \leq C_2} \sup_{z \in \C\setminus Z} |m^{P1}(y,z_0, z)|  < \infty.
\end{align}
\begin{figure}
\begin{center}
 \begin{overpic}[width=.55\textwidth]{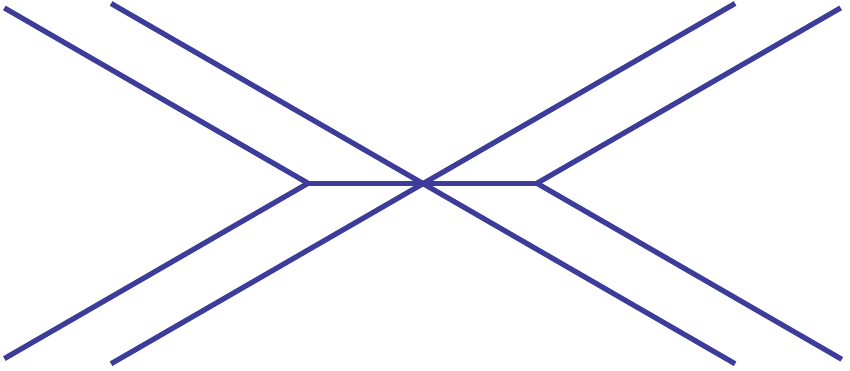}
 \put(77,34){\small $V_1$}
 \put(18,34){\small $V_2$}
 \put(18.5,8){\small $V_3$}
 \put(77,8){\small $V_4$}
 \put(62,18){\small $z_0$}
 \put(49,17.5){\small $0$}
 \put(33.5,18){\small $-z_0$}
 \end{overpic}
   \begin{figuretext}\label{Zfourrays.pdf}
      The regions $V_{1},...,V_{4}$
      \end{figuretext}
   \end{center}
\end{figure}
The function $m^Z$ satisfies (\ref{RHmZIVg}) iff $\hat{m}^Z := m^Z (m^{P1})^{-1}$ satisfies
\begin{align}\label{RHmZhatIVg}
\begin{cases}
\hat{m}^Z(y, t, z_0, \cdot) \in I + \dot{E}^2(\C \setminus Z),\\
\hat{m}_+^Z(y, t, z_0, z) =  \hat{m}_-^Z(y, t, z_0, z) \hat{v}^Z(y, t, z_0, z) \quad \text{for a.e.} \ z \in Z,
\end{cases}
\end{align}
where $\hat{w}^Z := \hat{v}^Z -I = m_{-}^{P1} (v^Z - v^{P1})(m_+^{P1})^{-1}$. 
For $z \in \cup_{k=1}^4 Z_k$ we have
\begin{align}\label{wZexpansionIV}
\hat{w}^Z(y, t, z_0, z) = \frac{\hat{w}_1^Z(y,z_0, z)}{t^{1/3}} + \cdots + \frac{\hat{w}_{n}^Z(y,z_0, z)}{t^{n/3}},
\end{align}
where 
\begin{align}\label{wjZIV}
\hat{w}_j^Z(y, z_0, z) = \begin{cases}
m_-^{P1} \begin{pmatrix} 
 0	& 0 \\
p_{j}z^je^{2i(y z + \frac{4z^3}{3})}  & 0 
\end{pmatrix}(m_+^{P1})^{-1}, &  z \in Z_1 \cup Z_2,
	\\
m_-^{P1}\begin{pmatrix} 0 & - \bar{p}_{j}z^je^{-2i(y z + \frac{4z^3}{3})} 	\\
0	& 0 
\end{pmatrix}(m_+^{P1})^{-1}, &   z \in Z_3 \cup Z_4.
\end{cases}
\end{align}
On the other hand, for $z \in Z_5$ we have
\begin{align}\label{wZexpansionIVZ5}
\hat{w}^Z(y, t, z_0, z) = \frac{\hat{w}_1^Z(y,z_0, z)}{t^{1/3}} + \cdots + \frac{\hat{w}_{2n}^Z(y,z_0, z)}{t^{2n/3}},
\end{align}
where (since $|e^{\pm 2i(y z + 4z^3/3)}| = 1$ for $z \in Z_5$)
\begin{align}\label{hatwjonZ5IVg}
|\hat{w}_j^Z(y, z_0, z)| \leq C, \qquad j = 1, \dots, 2n,
\end{align}
uniformly for $z \in Z_5$, $-C_1 \leq y \leq 0$, and $0 \leq z_0 \leq C_2$.

For $z = z_0 + re^{\frac{\pi i}{6}} \in Z_1$ with $r \geq 0$,  $z_0 \geq 0$, and $-4z_0^2 \leq y \leq 0$, we have
\begin{align*}
\re\bigg(2i\bigg(yz + \frac{4z^3}{3}\bigg)\bigg)
= -\frac{8 r^3}{3}-4 \sqrt{3} r^2 z_0-r \left(y+4 z_0^2\right)
\leq  -\frac{8 r^3}{3}-4 \sqrt{3} r^2 z_0.
\end{align*}
Hence the following estimate holds uniformly for $(y,t,z_0) \in \mathcal{P}_1$, where $\mathcal{P}_1$ denotes the parameter subset defined in (\ref{parametersetIVg}) with $T = 1$:
\begin{subequations}\label{eonZIVg}
\begin{align}
|e^{2i(y z + \frac{4z^3}{3})}| \leq Ce^{-|z-z_0|^2(z_0 + |z-z_0|)}, \qquad  z \in Z_1.
\end{align}
Analogous estimates hold also for $z \in Z_j$, $j = 2,3,4$:
\begin{align}
& |e^{2i(y z + \frac{4z^3}{3})}| \leq Ce^{-|z+z_0|^2(z_0 + |z+z_0|)}, \qquad  z \in Z_2,
	\\
& |e^{-2i(y z + \frac{4z^3}{3})}| \leq Ce^{-|z+z_0|^2(z_0 + |z+z_0|)}, \qquad  z \in Z_3,
	\\
& |e^{-2i(y z + \frac{4z^3}{3})}| \leq Ce^{-|z-z_0|^2(z_0 + |z-z_0|)}, \qquad  z \in Z_4.
\end{align}
\end{subequations}
Using (\ref{mP1bounded}) and (\ref{eonZIVg}), equation (\ref{wjZIV}) implies that, for any integer $m \geq 0$,
\begin{align*}
|z^m\hat{w}_j^Z(y,z_0, z)| \leq &\; C |z|^{m+j} e^{-|z-z_0|^3} \leq Ce^{-c|z-z_0|^3}, \qquad
 z \in Z_1, \ j = 1, \dots, n,
\end{align*}
uniformly for $(y,t,z_0) \in \mathcal{P}_1$. Similar estimates hold also on $Z_2 \cup Z_3 \cup Z_4$. 
Recalling also (\ref{hatwjonZ5IVg}), we conclude that, for any integer $m \geq 0$ and any $1 \leq p \leq \infty$,
\begin{align}\label{wZLpestIVg}
\begin{cases}
\|z^m\hat{w}_j^Z(y,z_0,z)\|_{L^p(Z)} \leq C, & j = 1, \dots, 2n,
	\\
\|z^m\hat{w}^Z(y,t,z_0,z)\|_{L^p(Z)} \leq Ct^{-1/3},
\end{cases} 
\end{align}
uniformly for $(y,t,z_0) \in \mathcal{P}_1$.
The estimates (\ref{wZLpestIVg}) play the same role in the present proof as the estimates (\ref{wYLpestIV}) did in the proof of Lemma \ref{YlemmaIV} and the rest of the proof is now analogous to the proof of that lemma.
\end{proof}

\noindent
{\bf Acknowledgement}  {\it The authors are grateful to the anonymous referee for valuable remarks. Support is acknowledged from the G\"oran Gustafsson Foundation, the European Research Council, Grant Agreement No. 682537, and the Swedish Research Council, Grant No. 2015-05430.}

\bibliographystyle{plain}
\bibliography{is}

\end{document}